\documentclass{amsart}
\usepackage[centertags]{amsmath}
\usepackage{amsfonts}
\usepackage{amssymb}
\usepackage{amsthm,mathtools}
\usepackage{amsrefs}

\usepackage[english]{babel}

\makeatletter
\g@addto@macro\bfseries{\boldmath} 
\makeatother

\usepackage[dvipsnames]{xcolor}

\usepackage[breaklinks=true,colorlinks=true,
linkcolor=MidnightBlue,citecolor=MidnightBlue,
urlcolor=MidnightBlue]{hyperref}

\usepackage{tikz}
\usepackage{tikz-cd}
\usepackage[normalem]{ulem}

\usepackage[shortlabels]{enumitem} 
\usepackage{bm}

\setlist[enumerate,1]{label=\textup{(\alph*)}}

\newcommand{\str}{\operatorname{str-exp}}
\newcommand{\smo}{\operatorname{smo}}
\newcommand{\aconv}{\operatorname{aconv}}
\newcommand{\conv}{\operatorname{conv}}
\DeclareMathOperator*{\esssup}{ess\,sup}

\newcommand{\K}{\mathbb{K}}
\newcommand{\N}{\mathbb{N}}

\newcommand{\Lin}{\mathcal{L}}
\newcommand{\ELL}{\mathcal{L}}

\newcommand{\restricted}{\mathord{\upharpoonright}}

\newcommand{\eps}{\varepsilon}
\newcommand{\cl}[1]{\overline{#1}}                          

\newcommand\restr[2]{{
    \left.\kern-\nulldelimiterspace 
      #1 
      \littletaller 
    \right|_{#2} 
  }}
\newcommand{\littletaller}{\mathchoice{\vphantom{\big|}}{}{}{}}

\DeclareMathOperator{\ACK}{ACK}

\DeclareMathOperator{\dist}{dist}
\DeclareMathOperator{\sgn}{sgn}
\DeclareMathOperator{\re}{Re}
\DeclareMathOperator{\spann}{span}
\DeclareMathOperator{\diam}{diam}

\DeclareMathOperator{\DP}{\mathcal{DP}}
\DeclareMathOperator{\NA}{NA}
\DeclareMathOperator{\QNA}{QNA}
\DeclareMathOperator{\RSE}{RSE}

\DeclareMathOperator{\UQNA}{UQNA}
\DeclareMathOperator{\ASE}{ASE}
\DeclareMathOperator{\FR}{\mathcal{FR}}
\DeclareMathOperator{\SE}{SE}
\DeclareMathOperator{\comp}{\mathcal{K}}

\newcommand{\essup}{\operatorname{ess~sup}}

\theoremstyle{plain}
\newtheorem{theorem}{Theorem}[section]
\newtheorem{cor}[theorem]{Corollary}
\newtheorem{corollary}[theorem]{Corollary}
\newtheorem{prop}[theorem]{Proposition}
\newtheorem{lemma}[theorem]{Lemma}
\newtheorem{lem}[theorem]{Lemma}

\theoremstyle{definition}
\newtheorem{example}[theorem]{Example}
\newtheorem{rem}[theorem]{Remark}
\newtheorem{remark}[theorem]{Remark}
\newtheorem{definition}[theorem]{Definition}

\begin{document}
\title[Range strongly exposing operators]{Range strongly exposing operators between Banach spaces} 

\author[Choi]{Geunsu Choi}
\address[Choi]{Department of Mathematics Education, Sunchon National University, 57922 Jeonnam, Republic of Korea \newline
\href{http://orcid.org/0000-0002-4321-1524}{ORCID: \texttt{0000-0002-4321-1524}}}
\email[Choi]{\texttt{gschoi@scnu.ac.kr}}

\author[del R\'{\i}o]{Helena del R\'{\i}o} 
\address[Del R\'{\i}o]{Department of Mathematical Analysis and Institute of Mathematics (IMAG), University of Granada, E-18071 Granada, Spain \newline
	\href{https://orcid.org/0009-0004-5078-6993}{ORCID: \texttt{0009-0004-5078-6993} }}
\email[Del R\'{\i}o]{\texttt{helenadelrio@ugr.es}}

\author[Fovelle]{Audrey Fovelle} 
\address[Fovelle]{Department of Mathematical Analysis and Institute of Mathematics (IMAG), University of Granada, E-18071 Granada, Spain \newline
	\href{https://orcid.org/0009-0000-5454-9726}{ORCID: \texttt{0009-0000-5454-9726} }}
\email[Fovelle]{\texttt{audrey.fovelle@ugr.es}}

\author[Jung]{Mingu Jung} 
\address[Jung]{June E Huh Center for Mathematical Challenges, Korea Institute for Advanced Study, 02455 Seoul, Republic of Korea\newline
\href{https://orcid.org/0000-0003-2240-2855}{ORCID: \texttt{0000-0003-2240-2855}}}
\email[Jung]{jmingoo@kias.re.kr}

\author[Mart\'in]{Miguel Mart\'in}
\address[Mart\'{\i}n]{Department of Mathematical Analysis and Institute of Mathematics (IMAG), University of Granada, E-18071 Granada, Spain \newline
\href{http://orcid.org/0000-0003-4502-798X}{ORCID: \texttt{0000-0003-4502-798X}}}
\email[Mart\'{\i}n]{\texttt{mmartins@ugr.es}}

\date{\today}



\begin{abstract}
We introduce a new class of bounded linear operators, called \textit{range strongly exposing} (\textit{RSE}, in short) operators, which form a natural intermediate class: weaker than absolutely strongly exposing operators, yet stronger than both uniquely quasi norm-attaining and classical norm-attaining operators. Several foundational results on norm-attaining operators are extended to the RSE setting. Among our main contributions, we establish that for every infinite-dimensional Banach space $Y$, there exists a Banach space $X$ such that the RSE operators from $X$ to $Y$ are not dense—an RSE analogue of a result by Acosta (1999) which applies only when $Y$ is strictly convex. We also show that the Radon-Nikod\'ym property of $Y$ is sufficient to obtain that RSE operators from $L_1(\mu)$ to $Y$ are dense and that this is also necessary if $\mu$ is not purely atomic. This extends and sharpens classical results by Uhl (1976). As a consequence, we prove that the set of RSE operators between $L_1(\mu)$ and $L_1 (\nu)$ is dense if and only if at least one of the measures $\mu$ or $\nu$ is purely atomic, in contrast with the classical result by Iwanik (1979) which guarantees the denseness of norm-attaining operators for all measures $\mu$ and $\nu$. We also prove that weakly compact operators from any $C(K)$ space can always be approximated by (weakly compact) RSE operators, thereby strengthening a result of Schachermayer (1983). Additionally, we present several improvements of more recent results concerning finite-rank operators and $\Gamma$-flat operators which give, in particular, RSE versions of classical results on compact operators by Johnson–-Wolfe (1979). Finally, we discuss RSE counterparts of results by Zizler and Lindenstrauss on the denseness of operators whose adjoints attain their norm.
\end{abstract}

\maketitle
{\parskip=0ex 	\tableofcontents }

\section{Introduction}

Let $X$ and $Y$ be Banach spaces over the field $\mathbb{K}$ (which will be the field $\mathbb{R}$ of real numbers or the field $\mathbb{C}$ of complex numbers). Our notation is standard, as in most textbooks on Banach space theory; we refer, e.g., to \cites{FHHMZ, megginson}. For example, we denote by $B_X$ and $S_X$ the closed unit ball and the unit sphere of $X$, respectively, and $\mathbb{T}:=S_{\mathbb{K}}$ is the set of modulus one scalars, $\Lin(X,Y)$ is the space of all bounded linear operators from $X$ into $Y$, and $X^*=\Lin(X,\K)$ denotes the (topological) dual space of $X$. Most of the standard concepts from Banach space theory used in the paper (such as rotundity, smoothness, and various operator ideals) can be found in Subsection~\ref{subsec:needed-definitions} at the end of this introduction.

It is said that $T \in \Lin(X,Y)$ \emph{attains its norm} (or that $T$ is \emph{norm-attaining}) if the operator norm $\|T\|$ coincides with $\max \{\|Tx\|\colon x \in B_X\}$, that is, if $T(B_X)\cap \|T\|S_Y\neq \emptyset$. We denote by $\NA(X,Y)$ the set of norm-attaining operators from $X$ into $Y$, and write $\NA(X)$ when $Y = \mathbb{K}$ for simplicity. The study of norm-attaining operators goes back to the seminal paper by Lindenstrauss \cite{Lin}, who proved, among other results, that there are Banach spaces $X$ and $Y$ for which $\NA(X,Y)$ fails to be dense in $\Lin(X,Y)$, and that $\NA(X,Y)$ is dense if either $X$ is reflexive or if $Y$ has a geometrical property called Lindenstrauss' property $\beta$ (which is satisfied by finite-dimensional polyhedral spaces and by closed subspaces of $\ell_\infty$ containing the canonical copy of $c_0$, among many other examples). To simplify the notation, Lindenstrauss introduced the following terminology. A Banach space $X$ has \textit{property A} if $\NA(X,Y)$ is dense in $\Lin(X,Y)$ for every target space $Y$. Analogously, a Banach space $Y$ has \textit{property B} if $\NA(X,Y)$ is dense in $\Lin(X,Y)$ for every domain space $X$. In 1977, Bourgain linked in \cite{Bourgain77} the study of norm-attaining operators with the Radon-Nikod\'ym property (for short, RNP) by showing that the RNP implies property A and presented a sort of reciprocal result. Huff improved the result further in \cite{Huff80} by showing that $X$ must have the RNP if any renorming of $X$ has property A. We refer the reader to the expository papers \cites{acosta-survey, ChoiKimLee-survey2024, Martin-RACSAM2016} and the recent papers \cites{Bachir, CJ23, Fovelle, JMR23,JungMartinezRueda-2024} and references therein for more information and background. Let us mention that both properties A and B are of isometric nature, that is, depend on the concrete norm we are using in the involved spaces, and that every separable Banach space can be renormed to have property A and every Banach space can be renormed to have property B.

To relate Bourgain's work with the concept we are introducing in this paper, let us describe it more in detail. What was actually proved in \cite{Bourgain77} is that the set of all absolutely strongly exposing operators (denoted by $\ASE(X,Y)$) is dense in $\Lin(X,Y)$ when $X$ has the RNP. Recall that $T\in \ASE(X,Y)$ if there exists $x_0 \in S_X$ such that whenever a sequence $(x_n) \subseteq B_X$ satisfies that $\|Tx_n\| \to \|T\|$, one can find a subsequence $(x_{\sigma(n)}) \subseteq (x_n)$ and $\theta \in \mathbb{T}$ such that $x_{\sigma(n)} \to \theta x_0$. That is, $T$ attains the norm at a unique element (up to rotation) in the strong way that all sequences in $B_X$ approximating the norm have to converge (up to rotation) to this element. This is, of course, a very strong concept. For instance, the point at which $T$ attains its norm has to be a strongly exposed point of $B_X$ (immediately) and the denseness of $\ASE(X,Y)$ for a non-zero $Y$ implies the denseness of the set of strongly exposing functionals in $X^*$ \cite[Proposition 1.5]{JMR23} (see Subsection~\ref{subsec:needed-definitions} for the definitions of strongly exposed points and strongly exposing functionals). We refer the reader to the recent paper \cite{JMR23} for extensions of Bourgain's work. Although Bourgain's positive result is both strong and useful when $X$ has the RNP, the concept of absolutely strongly exposing operator turns out to be highly restrictive when the space $X$ fails the RNP and is entirely useless when $B_X$ has no strongly exposed points. This occurs in classical Banach spaces such as infinite-dimensional $C(K)$ spaces and $L_1(\mu)$ spaces when $\mu$ is atomless. 

Let us now turn to property B, the counterpart to property A for range spaces. This property is also connected to the RNP: If every equivalent renorming of $Y$ has property B, then $Y$ has the RNP \cites{Bourgain77, Huff80}. However, unlike property A, the converse is false; as observed by Gowers \cite{Gowers}, the Banach spaces $\ell_p$ fail to have property B for $1<p<\infty$. Actually, the same happens for $\ell_1$ \cite{acosta-L1}, for all strictly convex infinite-dimensional Banach spaces \cite{acosta1}, and for $\ell_p$-sums of finite-dimensional Banach spaces for $1<p<\infty$ \cite{Fovelle}. Nevertheless, there remains a way to connect the RNP of a target space with its role as a `universal range' for the density of a kind of norm-attaining operators in a weaker sense. For this purpose, the notion of quasi norm-attainment was introduced in \cite{CCJM} as follows. An operator $T\in \Lin(X,Y)$ \emph{quasi attains its norm} (we write $T \in \QNA(X,Y)$) if there exists a sequence $(x_n) \subseteq B_X$ and $y_0 \in \|T\|S_Y$ such that $Tx_n \to y_0$ (that is, if $\overline{T(B_X)}\cap \|T\|S_Y\neq \emptyset$). It was shown in the aforementioned paper \cite{CCJM} that if $Y$ has the RNP, then $\QNA(X,Y)$ is dense in $\Lin(X,Y)$ for every Banach space $X$ and, conversely, the RNP is necessary to get that $\QNA(X,Y')$ is dense in $\Lin(X,Y')$ for every Banach space $X$ and every equivalent renorming $Y'$ of $Y$. It was actually proved in \cite{CCJM} that if $T \in \Lin (X,Y)$ is a strong Radon-Nikod\'ym operator (see the definition in Subsection~\ref{subsec:needed-definitions}, observe that in the case when $Y$ has the RNP, then every operator in $\Lin (X,Y)$ is strong Radon-Nikod\'ym), then $T$ can be approximated by uniquely quasi norm-attaining operators. We say that an operator $T\in \Lin(X,Y)$  \emph{uniquely quasi attains its norm} (and write $T \in \UQNA(X,Y)$) if there exists $y_0 \in \|T\|S_Y$ such that whenever a sequence $(x_n) \subseteq B_X$ satisfies $\|Tx_n\| \to \|T\|$, one can find a subsequence $(x_{\sigma(n)}) \subseteq (x_n)$ and $\theta \in \mathbb{T}$ such that $Tx_{\sigma(n)} \to \theta y_0$. Assuming that either the domain space $X$ or the range space $Y$ has the RNP (or that we are dealing with compact or weakly compact operators), the set of uniquely quasi norm-attaining (respectively, compact or weakly compact) operators between $X$ and $Y$ is dense in $\Lin (X,Y)$ \cite[Corollaries 3.10 and 3.11]{CCJM}. This notion is less useful than the one of absolutely strongly exposing operator, as it does not imply the operator to attain its norm. Instead, it provides a ``supremum value" that is unique (up to rotations) rather than an actual ``maximum value".

In this paper, we introduce the following notion of norm attainment, which is weaker than being absolutely strongly exposing, but stronger than the mere norm attainment and than being $\UQNA$. 

\begin{definition}
A bounded linear operator $T \in \Lin(X,Y)$ is said to \emph{strongly exposes its range} (or that it is \emph{range strongly exposing}, writing $T\in\RSE(X,Y)$), if there exists $x_0 \in S_X$ such that whenever $(x_n) \subseteq B_X$ satisfies  $\|Tx_n\| \to \|T\|$, one can find a subsequence $(x_{\sigma(n)}) \subseteq (x_n)$ and $\theta \in \mathbb{T}$ such that $Tx_{\sigma(n)} \to \theta Tx_0$. Sometimes, we will simply say that $T \in \Lin(X,Y)$ is an $\RSE$ operator when $T$ belongs to the class $\RSE(X,Y)$.
\end{definition}

The motivation behind the terminology ``$\RSE$'' stems from the following geometric observation: If $T \in \RSE(X,Y)$ witnessed by $x_0 \in S_X$, then $Tx_0$ is indeed a strongly exposed point of $\cl{T(B_X)}$ (see Lemma \ref{lem:strexp}). This strongly exposed point lies in the range space, and therefore the letter R in $\RSE$ is understood to signify the range space. As expected, the letters ``SE" stand for ``strongly exposing''. It will turn out that being a range strongly exposing operator is equivalent to being both uniquely quasi norm-attaining and norm-attaining (in the classical sense) simultaneously, Proposition \ref{prop:una=na+uqna}. Let us mention that the denseness of both $\NA(X,Y)$ and $\UQNA(X,Y)$ does not guarantee the denseness of $\RSE(X,Y)$, see Example~\ref{example:beta+RNP-notimplyRSEdense}. Observe that while an operator in the $\RSE$ class may attain its norm at many elements of the domain space, the operator has a unique (up to rotation) ``strong maximum value'' (in the sense that the image of any sequence approximating the norm of the operator has a subsequence which converges to a rotation of the maximum value).

The aim of this paper is to compare this new class of operators with previously studied notions of norm-attaining operators and to deepen the understanding of the interaction between the geometry of Banach spaces and norm-attaining theory. As we will see, the denseness of $\RSE$ operators has stronger consequences than the denseness of the norm-attaining operators.

Let us present the outline of the paper together with the main results obtained. The introduction concludes with Subsection \ref{subsec:needed-definitions} containing all the terminology we will use along the paper. 

Section \ref{sec:pre} presents all the preliminary and technical results on $\RSE$ and $\UQNA$ operators which we will use throughout the paper.

In Section \ref{sec:classical-results} we provide an $\RSE$ analogues of several classical results of norm-attaining operators. We first show in Subsection \ref{subsec:Acosta-RSE} that for every in\-fi\-nite-di\-men\-sion\-al Banach space $Y$, there is a Banach space $X$ such that $\RSE(X,Y)$ is not dense, that is, no infinite-dimensional Banach space satisfies the $\RSE$  version of Lindenstrauss' property B. Recall that every Banach space can be equivalently renormed to have Lindenstrauss' property $\beta$ \cite{Partington}, hence Lindenstrauss' property B. Subsection \ref{subsec:L1} deals with operators acting from $L_1(\mu)$ spaces, refining and extending the results of Uhl in 1976 \cite{Uhl}. Namely, we show that strong Radon-Nikod\'ym (in particular, weakly compact) operators from $L_1(\mu)$ spaces can be approximated by $\RSE$ operators. Conversely, we show that the denseness of $\RSE(L_1(\mu),Y)$ in $\ELL(L_1(\mu),Y)$ for a non purely atomic finite measure $\mu$, implies that $Y$ has the RNP. In contrast, the analogous result for norm-attaining operators holds only when $Y$ is strictly convex (for instance, $\NA(L_1(\mu),c_0)$ is dense for every measure $\mu$). As a consequence, we fully characterize when $\RSE(L_1(\mu),L_1(\nu))$ is dense by showing that it is the case exactly when at least one of the measures $\mu$ or $\nu$ is purely atomic---unlike the case of norm-attaining operators, where $\NA(L_1(\mu),L_1(\nu))$ is dense for all measures $\mu$ and $\nu$ \cite{Iwanik}. Next, we prove in Subsection \ref{subsec:W-on-CK} that weakly compact operator acting from a $C(K)$ space can be approximated by (weakly compact) $\RSE$ operators, improving a classical result by Schachermayer \cite{Schachermayer} of 1983. Finally, we extend the classical results of Johnson--Wolfe \cite{JW79} on norm-attaining compact operators to the $\RSE$ version in Subsection \ref{subsec:finiterank-compact}. In fact, we follow the ideas given in \cite{KLMW} to actually deal with finite-rank operators. The main result here shows that finite-rank operators defined on Banach spaces with ``sufficiently many'' linear subspaces inside their set of norm-attaining functionals can be approximated by $\RSE$ operators, a result that covers the case of compact operators from $C(K)$ and $L_1(\mu)$ spaces (by virtue of the approximation properties of their dual spaces) and many more examples. Let us recall that there are compact operators which cannot be approximated by norm attaining ones \cite{MarJFA20214}, and that the question of whether finite-rank operators can be always approximates by norm attaining ones is an old open problem \cite{KLMW}.

The aim of Section \ref{sec:ACK} is to extend the recent results in \cite{CJ23} on range spaces which have quasi-ACK structure, hence extending the results on ACK-structure of \cite{CGKS2018} (which include $C(K)$ spaces, uniform algebras, and many more spaces) and on the property quasi-$\beta$ of \cite{AAP2} (in particular, for Lindenstrauss' property $\beta$), see the paragraph preceding Theorem \ref{theo:smooth_dense}. We get a very general result on the denseness of $\RSE$ operators when the range space has quasi-ACK structure and the domain space satisfies that $\smo ({X^*}) \cap \NA(X)$ is dense in $X^*$ (this happens, for instance, when $X$ is strictly convex). As a consequence, if $\smo ({X^*}) \cap \NA(X)$ is dense in $X^*$ and $Y$ has property quasi-$\beta$ (for instance, if $Y$ is a closed subspace of $c_0(\Gamma)$), then every ideal $\mathcal{I}$ contained in the ideal of Dunford--Pettis operators is also contained in the closure of the set of $\RSE$ operators intersected with $\mathcal{I}$. Also, if $\smo ({X^*}) \cap \NA(X)$ is dense in $X^*$ and $Y$ is a predual of an $L_1$ space, then compact operators from $X$ to $Y$ can be approximated by $\RSE$ compact operators. Observe that all the results in this section have hypotheses in both the domain and the range space, but this is needed when the range space is infinite-dimensional as shown by Theorem~\ref{theorem:no-infdim-hasRSE-B} and Proposition~\ref{prop:no-universal-RSE-B-compact}. Let us comment that even in the case when the range space is a finite-dimensional $\ell_\infty$ space, we do not know if $\RSE$ operators are always dense without additional assumptions on the domain space.

Section \ref{sec:Zizler} deal with operators whose adjoints are $\RSE$. We first note that there is no $\RSE$ analogue of Zizler's (resp.\ Lindenstrauss's) results on the denseness of operators whose first (resp.\ second) adjoint attain their norm. Nevertheless, we show that Asplund operators with closed range can be approximated by Asplund operators whose adjoints belong to $\RSE$. In particular, finite-rank operators can always be approximated by finite-rank operators whose adjoints belong to $\RSE$ (so, compact operators can also be approximated by such operators under suitable approximation property conditions).

\subsection{Definitions of Key Concepts Used in the Paper}\label{subsec:needed-definitions}
We include here the definitions and basic properties of the concepts that we will use in the paper, even though most of them are nowadays well known. We refer to \cites{FHHMZ, megginson} for background. In all this subsection, $X$ and $Y$ will be Banach spaces. We will write $\operatorname{ext}(C)$ to denote the set of extreme points of $C$.

For a nonempty bounded subset $C$ of $X$, a point $x_0 \in C$ is called a \emph{strongly exposed point of $C$} if there is $x^* \in X^*\setminus \{0\}$ such that $\re x^*(x_0) = \sup_{x \in C} \re x^* (x)$ and whenever a sequence $(x_n) \subseteq C$ satisfies that $\lim_n \re x^* (x_n) = x^* (x_0)$, then $(x_n)$ converges in norm to $x_0$. In this case, we say that $x^*$ \emph{strongly exposes} $x_0$ in $C$ and $x^*$ is said to be a \emph{strongly exposing functional} of $C$. We write $\str(C)$ and $\SE(C)$ for, respectively, the set of strongly exposed points of $C$ and the set of strongly exposing functionals of $C$.

A point $x \in X$ is said to be a \textit{locally uniformly rotund point} (\emph{LUR point}, for short) if whenever $(x_n)$ is a sequence in $X$ such that $\|x_n\| \leq \|x\|$ for every $n \in \mathbb{N}$ and $\|x_n + x \| \to 2 \|x\|$, then $\|x_n - x\| \to 0$. If all elements in $X$ (equivalently, in $S_X$) are LUR points, then $X$ is called \textit{locally uniformly rotund} (\emph{LUR}, for short). It is immediate that a LUR point $x\in S_X$ is strongly exposed by all functionals in $X^*$ which attain their norm on $x$.

A closed convex subset $D$ of $X$ is said to be an \textit{Radon-Nikod\'ym property set} (RNP set, in short) if every subset of $D$ contains slices (i.e.\ nonempty intersections with open half-spaces) of arbitrary small diameter. A point $x_0$ of a closed bounded set $C$ is said to be a \textit{denting point} if for every $\varepsilon>0$ there exists a slice $S$ of $C$ such that $x_0\in S$ and $\diam (S)<\eps$. Clearly, every strongly exposed point of $C$ is also a denting point, while the opposite is not always true even in finite-dimensional spaces.

A point $x\in X$ is called a \textit{G\^{a}teaux} (resp.\ \textit{Fr\'echet}) \textit{differentiable point} if the norm is G\^{a}teaux (resp.\ Fr\'echet) differentiable at $x$, i.e., the limit 
\[
\lim_{t \to 0} \frac{\|x+th\|-\|x\|}{t}
\]
exists for every $h \in X$ (resp.\ uniformly $h \in S_X$). By the \v{S}mulyan's lemma (see for example \cite[Corollary 7.22]{FHHMZ}), a point $x \in X$ is a G\^{a}teaux (resp.\ Fr\'echet) differentiable point if and only if whenever $(f_n)$ and $(g_n)$ are sequences in $B_{X^*}$ such that $f_n(x) \to \|x\|$ and $g_n (x) \to \|x\|$, one has that the sequence $(f_n-g_n)$ converges to $0$ in the weak-star topology (resp.\ in norm topology). Besides, $x\in X$ is a G\^{a}teaux differentiable point if and only if there exists a unique $x^* \in S_{X^*}$ such that $x^* (x) = \|x\|$, and we also say that $x$ is a \textit{smooth point} of $X$. We write $\smo(X)$ for the set of smooth points in $X$. 

Let us now recall the different subsets of $\ELL(X,Y)$ that we will use all along the paper. The operator ideals of finite-rank operators, compact operators, and weakly compact operators will be denoted by $\FR(X,Y)$, $\comp(X,Y)$, and $\mathcal{W}(X,Y)$, respectively. We say that $T \in \Lin(X,Y)$ is an \textit{Radon-Nikod\'ym operator} (\emph{RN operator}, in short) if for every finite measure space $(\Omega, \Sigma, \mu)$ and every vector measure of bounded variation $G\colon\Sigma \to X$, the measure $T\circ G \colon\Sigma \to Y$ has a Bochner integrable Radon-Nikod\'ym derivative \cite{Edgar}. Also, recall $T \in \Lin (X,Y)$ is a \textit{strong Radon–Nikod\'ym operator} (\emph{strong RN operator}, in short) if $\overline{T (B_X)}$ is an RNP set. Let us denote by $\mathcal{SRN}(X,Y)$ the class of strong RN operators from $X$ into $Y$. Notice that every strong RN operator is an RN operator \cite{Edgar}; but the converse is not true in general \cite{GS84}. Recall also that $T \in \Lin (X,Y)$ is an \textit{Dunford-Pettis operator} if $T$ sends weakly convergent sequences to norm convergent sequences. Let us denote by $\DP(X,Y)$ the operator ideal of Dunford-Pettis operators. An operator $T\in\Lin(X,Y)$ is called an \emph{Asplund operator} if it factors through an Asplund space \cite{Edgar}; the collection of such operators is denoted by $\mathcal{A}(X,Y)$.

Finally, a Banach space $X$ is said to have the \textit{Radon-Nikod\'ym property with respect to a finite measure space $(\Omega,\Sigma,\mu)$} if each $\mu$-continuous vector measure $G\colon \Sigma \to X$ of bounded variation has a Bochner integrable Radon-Nikod\'ym derivative. If $X$ has the Radon-Nikod\'ym property with respect to every finite measure space, we say $X$ has the \textit{Radon-Nikod\'ym property} (RNP, in short). It is known that $X$ has the RNP if and only if $X$ has the RNP with resepct to the Lebesgue measure on $[0,1]$ \cite[Corollary V.3.8]{DU}. More generally, $X$ has the RNP if and only if $X$ has the RNP with respect to some finite measure space $(\Omega,\Sigma,\mu)$ where $\mu$ is not purely atomic \cite[Theorem~2]{Chatterji}. Let us remark that $X$ has the RNP if and only if the identity operator on $X$ is a (strong) RN operator, which is also equivalent to the unit ball $B_X$ being an RNP set. 

\section{Preliminary Results}\label{sec:pre}
Our first aim in this section is to discuss the relation between all the notions of norm attainment that we have presented in the introduction. We start with a diagram showing the implications between the properties.
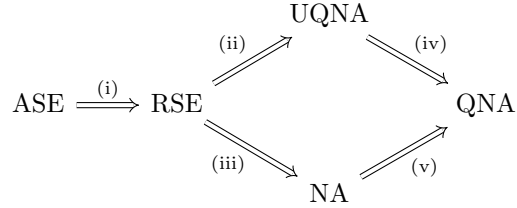
\begin{figure}[ht!]
\centering
\[
\begin{tikzcd}
	&& \text{UQNA} \\
	\text{ASE} & \RSE && \text{QNA} \\
	&& \text{NA}
	\arrow["\textup{(iv)}", Rightarrow, from=1-3, to=2-4] 
	\arrow["\textup{(i)}", Rightarrow, from=2-1, to=2-2] 
	\arrow["\textup{(ii)}", Rightarrow, from=2-2, to=1-3] 
	\arrow["\textup{(iii)}"', Rightarrow, from=2-2, to=3-3] 
	\arrow["\textup{(v)}"', Rightarrow, from=3-3, to=2-4] 
\end{tikzcd}
\]
\caption{Relation between the different kinds of norm attainment}
\label{fig:relation}
\end{figure}

We next emphasise that none of the implications above can be reversed in general or, equivalently, that the sets in the above diagram are different. Even more, we can find examples for each case in which the bigger set is dense while the smaller is not.

\begin{remark}\label{remark:implications}
None of the implications in the above diagram can be reversed in general, and neither $\NA(X,Y)$ is contained in the closure of $\UQNA(X,Y)$ nor $\UQNA(X,Y)$ is contained in the closure of $\NA(X,Y)$ . 
\begin{itemize}
\itemsep0.25em
\item Notice that $\RSE(X, \mathbb{K}) = \NA(X,\mathbb{K})$ and $\UQNA(X,\mathbb{K}) = \QNA(X,\mathbb{K})=\ELL(X,\mathbb{K})$ for every Banach space $X$. Thus, in particular, $\RSE(c_0,\mathbb{K})$ is dense in $\Lin(c_0,\mathbb{K})$, while the set $\ASE(c_0,\mathbb{K})$ is trivial (since there is no strongly exposed point of $B_{c_0}$); this shows that (i) cannot be reversed.
\item  It was shown by Gowers \cite{Gowers} that there is a Banach space $G$ such that the set $\NA(G,\ell_2)$ is not dense in $\Lin (G, \ell_2)$. Therefore, the set $\RSE(G, \ell_2)$ is not dense in $\Lin (G, \ell_2)$. On the other hand, since $\ell_2$ has the RNP, the set $\UQNA(G,\ell_2)$ is dense in $\Lin(G,\ell_2)$ by \cite[Corollary 3.10]{CCJM}. Thus, none of the implications (ii) and (v) can be reversed. Besides, it also follows that $\UQNA(X,Y)$ is not always contained in the closure of $\NA(X,Y)$.
\item The set $\NA(c_0,c_0)$ is dense in $\Lin(c_0,c_0)$ by an old result of Lindenstrauss \cite[Proposition 3]{Lin}, while the identity operator on $c_0$ does not belong to the norm closure of $\UQNA (c_0, c_0)$ \cite[Example 7.8]{CCJM}. So, (iii) and (iv) cannot be reversed. It is also shown that $\NA(X,Y)$ is not always contained in the closure of $\UQNA(X,Y)$.
\end{itemize} 
\end{remark}

The next example also shows that the implication (iii) cannot be reversed even when both the domain and range spaces are of finite dimension (greater than one):

\begin{example}
For every $n \geq 2$, let $Y$ be a Banach space with $\dim Y\geq 2$. Then, $$\RSE(\ell_1^n, Y) \neq \ELL(\ell_1^n , Y)=\NA(\ell_1^n,Y).$$ Indeed, take two linearly independent vectors $u, v \in S_Y$ such that for any $\alpha, \beta \in \mathbb{K}$ with $|\alpha|+|\beta| = 1$, $\| \alpha u + \beta v \| = 1$ if and only if $\alpha =0$ or $\beta =0$ (this is clearly possible in every Banach space of dimension at least two). Let $T\colon \ell_1^n \rightarrow Y$ be defined by $T(e_1)=u$, $T(e_2) = v$, and $T(e_j)=0$ for $j >2$. Then, it is clear that $T \not\in \RSE( \ell_1^n, Y)$. 
\end{example} 

Nevertheless, there is a significant relationship between the properties in Figure~\ref{fig:relation}: being $\RSE$ is equivalent to simultaneously belonging to the classes $\NA$ and $\UQNA$.

\begin{prop}\label{prop:una=na+uqna}
Let $X$ and $Y$ be Banach spaces. Then, 
\[
\RSE(X,Y) = \NA(X,Y) \cap \UQNA(X,Y).
\]
\end{prop}

\begin{proof}
One way is clear. Let $T \in \NA(X,Y) \cap \UQNA(X,Y)$. Then, there exists $x_0 \in S_X$ such that $\|Tx_0\|=\|T\|$. Let $(x_n) \subseteq B_X$ be any subsequence satisfying $\|Tx_n\| \to \|T\|$. Since $T \in \UQNA(X,Y)$, there exists $y_0 \in \|T\|S_Y$, $\theta \in \mathbb{T}$, and a subsequence $(x_{\sigma(n)}) \subseteq (x_n)$ such that $Tx_{\sigma(n)} \to \theta y_0$. Similarly, for the constant sequence $(x_0) \subseteq B_X$, we have $Tx_0 \to \theta' y_0$ for some $\theta' \in \mathbb{T}$. This shows that $Tx_{\sigma(n)} \to \theta \overline{\theta'} Tx_0$, hence $T \in \RSE(X,Y)$.
\end{proof}

Let us remark that Example~\ref{example:beta+RNP-notimplyRSEdense}, which appears in the next section, will show that the denseness of both $\NA(X,Y)$ and $\UQNA(X,Y)$ does not guarantee the denseness of $\RSE(X,Y)$.

The next observation provides some sufficient conditions for the reverse implications of (i), (ii), and (v) in Figure~\ref{fig:relation} to hold. Item (a) below appeared already in \cite[Remark 1.5]{CCJM} and we include it for the sake of completeness.

\begin{rem}\label{rem:monomorf}
Let $X$, $Y$ be Banach spaces and let $T\in \ELL(X, Y)$ satisfy that $T(B_X)$ is closed. \begin{enumerate}
\itemsep0.25em
\item $T\in \QNA(X, Y)$ if and only if  $T\in \NA(X, Y)$; 
\item $T\in \UQNA(X, Y)$ if and only if $T\in \RSE(X, Y)$. 
\end{enumerate}
If, in addition, $T$ is a monomorphism, then 
\begin{enumerate}
\itemsep0.25em
\item[(c)] $T\in \UQNA(X, Y)$ if and only if $T\in \ASE (X, Y)$. 
\end{enumerate}
\end{rem} 

\begin{proof}
    Indeed, if $T\in \QNA(X, Y)$ then $T(B_X)\cap \|T\|S_Y=\overline{T(B_X)}\cap \|T\|S_Y\neq \emptyset$, so $T$ attains its norm. Moreover, if $T\in \UQNA(X, Y)$ then $$T\in \UQNA(X, Y)\cap \NA(X, Y)=\RSE(X, Y).$$ 
    For (c), suppose that $T\in \UQNA(X, Y)$ and let $(x_n)\subseteq S_X$ be a sequence such that $\|Tx_n\|\to \|T\|$. Therefore, there exist $y_0\in \|T\|S_Y$ and a subsequence  $(x_{\sigma(n)})\subseteq (x_n)$ and $\theta \in \mathbb{T}$ such that
    \begin{equation*}
        Tx_{\sigma(n)}\to \theta y_0.
    \end{equation*}
Since $T$ is a monomorphism, we have $x_{\sigma(n)} \rightarrow x_0 := \theta \, T^{-1} (y_0)$.
    \end{proof}

Another way to pass from being range strongly exposing to being absolutely strongly exposing is to attain the norm at a LUR point.
 
\begin{prop}\label{prop:RSE+aLUR}
Let $X$ and $Y$ be Banach spaces. If $T\in \RSE(X, Y)$ satisfies that $\|Tx_0\| = \|T\|$ for  $x_0 \in S_X$ and $x_0$ is strongly exposed by every functional attaining its norm on it (for instance, if $x_0$ is a LUR point), then $T\in \ASE(X, Y)$. 
\end{prop}
   
\begin{proof}
For simplicity, we may assume that $\|T\|=1$. Suppose that $(x_n)\subseteq B_X$ is a sequence such that $\|Tx_n\|\to 1$. Then, there exists a subsequence $(x_{\sigma(n)})\subseteq (x_n)$ and $\theta\in \mathbb{T}$ such that $Tx_{\sigma(n)}\to \theta Tx_0$. Take $y_0^* \in S_{Y^*}$ so that $y_0^* (T x_0) = 1$. Then, $T^* y_0^* (x_0) =1$ and $T^* y_0^* \bigl( \overline{\theta} x_{\sigma(n)}\bigr) \to 1$. Since $x_0$ is strongly exposed by $T^* y_0^*$, we obtain that $( x_{\sigma(n)})$ converges to $\theta x_0$. 
\end{proof}

A useful method used in \cites{cgmr2020,JMR23} to obtain $\ASE$ operators is to consider bounded linear operators that attain their norm at strongly exposed points. Concretely, if $T \in \Lin(X, Y)$ attains its norm at $x_0 \in \str(B_X)$ and $\varepsilon > 0$, then there exists $S \in \ASE(X, Y)$ such that $\|Sx_0\| = \|S\|$ and $\|S - T\| < \varepsilon$ (see \cite[Proposition 3.14]{cgmr2020} and \cite[Lemma 1.2]{JMR23}). The following lemma follows a similar approach and will be useful in establishing the denseness of $\RSE(X, Y)$ and $\UQNA(X, Y)$ based on the denseness of $\NA(X, Y)$ and $\QNA(X, Y)$, respectively.

\begin{lem}\label{lem:strexp+density}
    Let $X$, $Y$ be Banach spaces and $T\in \ELL(X, Y)$.
    \begin{enumerate}
    \itemsep0.25em
        \item If $T\in \QNA(X, Y)$ towards $y_0\in \|T\|S_Y$ and $y_0 \in\str\bigl(\overline{T(B_X)}\bigr)$, then $T\in \overline{\UQNA(X, Y)}$. 
        \item If $T\in \NA(X, Y)$ and $\|Tx_0\|=\|T\|$ for some $x_0 \in S_X$ such that $Tx_0\in\str\bigl(\overline{T(B_X)}\bigr)$, then $T\in \overline{\RSE(X, Y)}$.
    \end{enumerate}
\end{lem}
\begin{proof}
(a) Fix $\varepsilon>0$ and let $y_0^* \in {Y^*}$ be a strongly exposing functional of $\overline{T(B_X)}$ at $y_0$ such that $\|y_0^*\|<\frac{1}{\|T\|}$. Observe that 
    \begin{equation}\label{eq:sup}
        \sup \bigl\{ |y_0^*(y)|\colon y\in\overline{T(B_X)} \bigr\} = |y_0^* (y_0)| = y_0^*(y_0).
    \end{equation}
Now, we define $S\in \ELL(X, Y)$ as
    \begin{equation}\label{eq:def_S}
        Sx= Tx + \varepsilon y_0^*(Tx)\frac{y_0}{\|y_0\|}.
    \end{equation}
    Then, we have $\|S-T\|\leq\varepsilon\|y_0^*\|  \|T\|<\varepsilon$ and, applying \eqref{eq:sup}, $$\|Sx\|\leq \|T\|+ \varepsilon|y_0^*(Tx)|\leq\|y_0\|+\varepsilon y_0^*(y_0)$$ 
    for every $x\in B_X$, so $\|S\|\leq \|y_0\|+\varepsilon y_0^*(y_0)$. In fact, we will prove that $\|S\|=\|y_0\|+\varepsilon y_0^*(y_0)$. 
    Choose a sequence $(x_n)\subseteq B_X$ such that $Tx_n \to y_0$, and we have 
    \begin{equation*}
        \lim_n Sx_n = y_0+\varepsilon y_0^*(y_0) \frac{y_0}{\|y_0\|}=\left(1+\varepsilon\frac{y_0^*(y_0)}{\|y_0\|}\right)y_0 =: z_0.
    \end{equation*}
    Observe that $\|z_0\|= \|y_0\|+\varepsilon y_0^*(y_0)$; so $\|S\| =  \|y_0\|+\varepsilon y_0^*(y_0)$. 
    
    Finally, to prove that $S\in \UQNA(X, Y)$ take a sequence $(u_n)\subseteq B_X$ such that $\|Su_n\|\to \|S\|$. Then by \eqref{eq:sup} 
    \begin{equation*}
     \|S\| =\lim_n \left\|Tu_n+\varepsilon y_0^*(Tu_n)\frac{y_0}{\|y_0\|}\right\|  \leq \|T\|+\varepsilon y_0^* (y_0).
    \end{equation*} 
In particular, this implies that $\lim_n|y_0^*(Tu_n)|= y_0^*(y_0)$; so there exists a sequence $(\theta_n)\subseteq \mathbb{T}$ such that $\re y_0^*(T(\theta_n u_n))\to y_0^*(y_0)$. As $y_0^*$ strongly exposes $\overline{T(B_X)}$ at $y_0$, we have $T(\theta_n u_n)\to y_0$, and then $S(\theta_n u_n)\to z_0$, as we wanted.
    
    (b) If $T\in \NA(X, Y)$ then, in particular, $T\in \QNA(X, Y)$. So, we can follow the proof of (a) and consider $S\in \UQNA(X, Y)$ as in \eqref{eq:def_S} with $y_0=Tx_0$. 
    Note that $\|S-T\|<\varepsilon$, and, in this case, $\|Sx_0\|=\|S\|$. 
    Thus, $S\in \NA(X, Y)$; therefore $S\in \RSE(X, Y)$ thanks to Proposition \ref{prop:una=na+uqna}.
\end{proof}

\begin{remark}\label{remark:comment-to-lemma26}
Let us comment that in the proof of the lemma above, in both cases we construct operators $S\in \RSE(X,Y)$ which are arbitrarily close to the given operator $T$, such that $T-S$ is of rank-one and satisfying that $\overline{S(B_X)}\subseteq \beta \overline{T(B_X)}$ for some $\beta>0$.
\end{remark}

As a direct consequence of Lemma~\ref{lem:strexp+density}, we have the following corollary:

\begin{corollary}\label{cor:lur}
Let $X$ and $Y$ be Banach spaces. Suppose that $S_Y = \str(B_Y)$ (in particular, if $Y$ is LUR). Then 
\begin{enumerate}
\itemsep0.25em
\item $\NA(X, Y) \subseteq \overline{\RSE(X, Y)}$;
\item $\QNA(X, Y)\subseteq \overline{\UQNA(X, Y)}$. 
\end{enumerate}
\end{corollary}

The next result complements the relationship between UQNA and RSE operators with the strongly exposed points of the (closure) of the image of the unit ball. Notice the parallelism with \cite[Lemma~1.2]{JMR23}.

\begin{lem}\label{lem:strexp}
    Let $X$ and $Y$ be Banach spaces. Then
    \begin{enumerate}
    \itemsep0.25em
        \item If $T\in \UQNA(X, Y)$ towards $y_0\in \|T\|S_Y$ then $y_0\in \str\bigl(\overline{T(B_X)}\bigr)$ and it is strongly exposed by any $y^*_0\in S_{Y^*}$ such that $\re y_0^*(y_0)=\|y_0\|$.
        \item If $T\in \RSE(X, Y)$ and $x_0\in S_X$ satisfies $\|Tx_0\|=\|T\|$, then $Tx_0\in \str\bigl(\overline{T(B_X)}\bigr)$ and it is strongly exposed by any $y_0^*\in S_{Y^*}$ such that $\re y_0^*(Tx_0)=\|T\|$.
    \end{enumerate}
\end{lem}
\begin{proof} (a) Let $y_0^*\in S_{Y^*}$ such that $\re y_0^*(y_0)=\|y_0\|=\|T\|$. Observe that we have
\begin{equation*}
    \|T\|=\re y_0^*(y_0) \leq \sup \bigl\{\re y_0^*(z) \colon z\in \overline{T(B_X)}\bigr\}\leq \|T\|
\end{equation*}
and, therefore, $$\sup \bigl\{\re y_0^*(z) \colon z\in \overline{T(B_X)}\bigr\}=\|T\|.$$  Now, suppose that $(y_n)\subseteq \overline{T(B_X)}$ is a sequence such that 
\begin{equation*}
    \lim_n \re y_0^*(y_n) =\sup \{\re y_0^*(z) \colon z\in \overline{T(B_X)}\}=\|T\|.
\end{equation*}
Find $x_n \in B_X$ for each $n \in \N$ so that $\| y_n- Tx_n\| \rightarrow 0$ as $n \rightarrow \infty$; thus we have 
\begin{equation}\label{eq:lim}
    \lim_n \re y_0^*(Tx_n) =\|T\|
\end{equation}
which, in particular, implies that $\|Tx_n\| \rightarrow \|T\|$ as $n \rightarrow \infty$. 
Thus, there exists a subsequence $(x_{\sigma(n)})\subseteq (x_n)$ and $\theta\in \mathbb{T}$ such that $Tx_{\sigma(n)}\to \theta y_0$. Applying \eqref{eq:lim}, we have that
\begin{equation*}
    \|T\|=\lim_n \re y_0^*(Tx_{\sigma(n)})=  \re \theta \, y_0^*(y_0) = \re \theta \, \|T\| 
\end{equation*}
so $\theta$ must be $1$. In fact, every subsequence of $(Tx_n)$ must contain a further subsequence converging to $y_0$, so we finally have $\|Tx_n - y_0\| \rightarrow 0$; hence $\|y_n - y_0\|\rightarrow 0$ as $n \rightarrow \infty$.

(b) The proof is almost identical with (a), by applying $\sup \{\re y_0^*(Tx) \colon x \in B_X\} = \|T\|$.
\end{proof}

For the next few results, let $T\in \ELL(X, Y)$ and consider the commutative diagram in Figure \ref{fig:canonical_factorization}, where the operator $q\colon X \to X /\ker T$ denotes the canonical quotient map sending $x \in X$ to $[x] := x + \ker T \in X/ \ker T$ and $\tilde{T}$ is the injectivization of $T$.
\begin{figure}[ht!]
\centering
\[
 \begin{tikzcd}
		X \arrow[rd,"q"'] \arrow[rr,"T"] &  & Y \\
		& X/\ker T \arrow[ru,"\tilde{T}"'] & 
	\end{tikzcd}
\]
\caption{Canonical factorization of $T\in\Lin(X,Y)$}
\label{fig:canonical_factorization}
\end{figure}

Let us mention that $T \in \UQNA(X,Y)$ if and only if $\tilde{T} \in \UQNA(X/\ker T, Y)$, a fact that which can be deduced from the argument in \cite[Lemma 5.6]{CCJM}. Moreover, if $T$ has  closed range (recall that $T$ has a closed range if and only if $\tilde{T}$ is a monomorphism), then $\tilde{T} \in \UQNA(X/\ker T, Y)$ is equivalent to $\tilde{T} \in \ASE(X / \ker T, Y)$ by Remark \ref{rem:monomorf}. Summarizing, we have the following: 

\begin{remark}\label{remark:factorization_UQNA}
    Given $T \in \Lin (X,Y)$, we have 
    \begin{equation*}
        T \in \UQNA(X,Y) \iff \tilde{T} \in \UQNA (X / \ker T, Y) \,  \xLeftarrow{(\diamondsuit)} \, \tilde{T} \in \ASE (X / \ker T, Y)
    \end{equation*}
    and the implication $(\diamondsuit)$ can be reversed when $T$ has closed range. 
\end{remark}

Two related important observations in the same line are the following.

\begin{prop}\label{prop:quotient_implication2}
Given $T \in \Lin (X,Y)$, $R\in \Lin(X/\ker T,Y)$, if $R \in \ASE(X/\ker T, Y)$  with  $\|R(q(x_0))\| = \|R\|$ for some $x_0 \in B_X$, then $R \circ q \in \RSE(X, Y)$. 
\end{prop}

\begin{proof}
If $(x_n)\subset B_X$ is such that $\|R(q(x_n))\|\to \|R\|$, then passing to a subsequence, $q(x_n)\to \theta q(x_0)$ for some $\theta \in \mathbb{T}$, and therefore $R(q(x_n))\to \theta R(q(x_0))$. This shows that $R\circ q\in \RSE(X, Y)$ at $x_0$.    
\end{proof}

The implication in Proposition \ref{prop:quotient_implication2} can be reversed for $R=\tilde{T}$ when $T$ has closed range. In fact, if $T = \tilde{T} \circ q \in \RSE(X,Y)$ then, in particular, $T \in \UQNA(X,Y)$. This is equivalent to $\tilde{T} \in \ASE(X/\ker T ,Y)$ as we observed in Remark \ref{remark:factorization_UQNA}. 

\begin{remark}\label{remark:proximinal_updated_remark}
    Notice that when $R \in \ASE(X/\ker T, Y) \text{ with } \|R(q(x_0))\| = \|R\|$ and $q(x_0) \in B_{X / \ker T}$, the point $x_0$ can be chosen from the closed unit ball $B_X$ in the following situations:
\begin{enumerate}
    \item  when $\ker T$ is proximinal, i.e., $q(B_X)=B_{X/\ker T}$. 
    \item  when $(\ker T)^\perp \subseteq \NA(X)$.
\end{enumerate}
In fact, the case (a) is somewhat obvious. As for the case (b), note that $\|R^*\| = \|R^* (y_0^*)\|$ for some $y_0^* \in B_{Y^*}$. Moreover, $x_0^* := (R\circ q)^* (y_0^*) = q^* (R^* (y_0^*))$ vanishes on $\ker T$. By hypothesis, this implies that $x_0^* \in \NA(X)$; so $|x_0^* (x_0)| = \|x_0^*\|$ for some $x_0 \in B_X$. It follows that $\|R\| = \| R(q(x_0))\|$.
\end{remark}

From the discussion above, we can deduce the following results regarding approximation by $\RSE$ operators. 

\begin{prop} \label{prop:proximinal+RNP}
Let $X$ and $Y$ be Banach spaces, and let $\mathcal{I}$ be an operator ideal. If $T \in \mathcal{I} (X,Y)$ satisfies one of the following conditions: 
\begin{enumerate}
\itemsep0.25em
    \item $\ker T$ is proximinal and $X / \ker T$ has the RNP;
    \item $(\ker T)^\perp \subseteq \NA(X)$,
\end{enumerate}
then $T \in \overline{\mathcal{I} (X,Y) \cap \RSE(X,Y)}$. 
\end{prop} 

\begin{proof}
(a): Fix $\varepsilon>0$. Since $X/\ker T$ has the RNP, there is $R \in \ASE (X/\ker T, Y)$ so that $R \in \Lin (X/\ker T, Y)$, $Q:=R-\tilde{T}$ is of finite-rank, and $\|R - \tilde{T} \| < \eps$, where $T = \tilde{T} \circ q$. Finally, notice from Remark \ref{remark:proximinal_updated_remark} that the proximinality of $\ker T$ yields that $R\circ q \in  \RSE(X,Y) $ with $\| R \circ q - T \| < \varepsilon$. Finally, note that $R\circ q = Q \circ q + T \in \mathcal{I}(X,Y)$. 

(b): Observe that $(\ker T)^\perp \subseteq \NA(X)$ implies that $X / \ker T$ is reflexive. Thus, as in (a), we can find $R \in  \ASE (X/\ker T, Y) $ so that $\|R - \tilde{T} \| \approx 0$. Finally, we deduce from the item (b) of Remark \ref{remark:proximinal_updated_remark} that $R\circ q \in \mathcal{I}(X,Y) \cap \RSE(X,Y)$ and $\|R\circ q - T \| \approx 0$.
\end{proof} 

In particular, item (a) of Proposition \ref{prop:proximinal+RNP} shows that if $T \in \Lin (X,Y)$ has closed range, $\ker T$ is proximinal, and $Y$ has the RNP, then $T \in \overline{\RSE(X,Y)}$. We note that item (b) improves \cite[Proposition 3.15]{KLMW}, which proves that $T \in \overline{\NA(X,Y)}$ under the same assumption.

We can also take advantage of the observations above to obtain a result similar to Lemma \ref{lem:strexp+density}.

\begin{prop}\label{prop:strexp-quotient} 
Let $X$ and $Y$ be Banach spaces, and let $\mathcal{I}$ be an operator ideal. If $T \in \mathcal{I} (X,Y)$ attains its norm at some $x_0 \in S_X$ with $q (x_0) \in \str (B_{X/\ker T})$, then $T \in \overline{\mathcal{I}(X,Y) \cap \RSE(X,Y)}$. 
\end{prop} 

\begin{proof} Let $\tilde{T} \in \Lin (X/\ker T, Y)$ be the operator induced by $T \in \mathcal{I}(X,Y)$. Since $\tilde{T}$ attains its norm at $q(x_0)$ which is a strongly exposed point, by \cite[Lemma 1.2]{JMR23}, there exists $R \in \ASE(X/\ker T, Y)$ such that $\| R (q(x_0))\| = \| R \|$, $Q:= R-\tilde{T}$ is of finite-rank, and $\|R - \tilde{T} \| < \eps$. So, $R\circ q \in \mathcal{I}(X,Y)$ and $\|R\circ q - T \| < \varepsilon$. Notice from Proposition \ref{prop:quotient_implication2} that $R \circ q \in \RSE(X,Y)$.
\end{proof}

As a consequence, we obtain the following result on finite-rank operators.

\begin{cor}\label{cor:strexp-quotient2} Let $X$ and $Y$ be Banach spaces, and $T \in \FR(X,Y) \cap \NA(X,Y)$. If $X/\ker T$ is strictly convex, then $T \in \overline{\FR(X,Y)\cap\RSE(X,Y)}$. 
Therefore, if $X^*$ is smooth, then 
\begin{equation*}
\FR(X,Y) \cap \NA(X,Y) \subseteq \overline{\FR(X,Y) \cap \RSE(X,Y)}. 
\end{equation*}
\end{cor} 

\begin{proof}
Let $T \in \FR(X,Y) \cap \NA (X,Y)$ be given. Since $X/\ker T$ is strictly convex, every point in $S_{X/\ker T}$ is an exposed point. By finite-dimensionality of $X/\ker T$, we conclude that $S_{X/\ker T} = \str (B_{X / \ker T})$. Proposition \ref{prop:strexp-quotient} shows that $T \in \overline{\FR(X,Y) \cap \RSE(X,Y)}$ and completes the proof.

For the last part, just observe that if $X^*$ is smooth, then every quotient of $X$ is strictly convex.
\end{proof}

The last goal in this section will be to show that $\UQNA$ is a $G_\delta$ subset of $\Lin (X,Y)$. Recall that the same is true for the set $\ASE(X,Y)$ (see, for instance, the proof of \cite[Theorem 5]{Bourgain77}). Let us mention that it can be easily observed that $\RSE (X,Y)$ is not always residual, even being dense. Indeed, note that $\RSE(c_0, \mathbb{K}) = \NA(c_0,\mathbb{K})=\ell_1 \cap c_{00}$ (cf.\ Remark \ref{remark:implications} and \cite[Example 1.9]{JMR23}). 

\begin{prop}\label{prop:gdelta}
Let $X$ and $Y$ be Banach spaces. Then, the set $\UQNA(X, Y)$ is a $G_{\delta}$ subset of $\ELL(X, Y)$.
\end{prop}

Before giving a proof of this result, we state the elementary lemma which can be proved easily. 

\begin{lem}\label{lem:Ae}
    Let $X$ and $Y$ be Banach spaces. For every $T\in \ELL(X, Y)$, define
    \begin{equation*}
        \widetilde{S}(T, \eta)=\{Tx \colon  x\in B_X \text{ with }  \|Tx\|>\|T\|-\eta\}.
    \end{equation*}
    Then, for every $\varepsilon>0$ the set
    \begin{equation*}
        A_{\varepsilon}=\{T\in \ELL(X, Y) \colon \exists \, y\in Y, \eta>0 \text{ such that } \, \widetilde{S}(T, \eta)\subseteq \mathbb{T} U(y, \varepsilon)\}
    \end{equation*}
    is open, where $U(y,\varepsilon)$ denotes the set $\{ z \in Y \colon \|z-y\| < \varepsilon\}$.
\end{lem}

\begin{proof}[Proof of Proposition \ref{prop:gdelta}]
    Consider, for each $\varepsilon>0$, the set $A_\varepsilon$ as in Lemma \ref{lem:Ae} which is an open subset in $\Lin (X,Y)$. 
    We will prove that $$\UQNA(X, Y) = \bigcap\nolimits_{n\in \mathbb{N}} A_{n^{-1}}.$$
    First, suppose that $T\in \UQNA(X, Y)$ towards $y_0\in \|T\|S_Y$. If we had $T\notin A_{n_0^{-1}}$ for some $n_0\in \mathbb{N}$, then in particular we would have 
    $$\widetilde{S} \left(T, k^{-1} \right)\not\subseteq \mathbb {T} U(y_0, n_0^{-1})$$ 
    for all $k\in \mathbb{N}$. In other words, there exists $z_k\in B_X$ such that $Tz_k\in\tilde{S}(T, \frac{1}{k})$ and $Tz_k \notin \mathbb {T} U(y_0, n_0^{-1})$ for every $k\in \mathbb{N}$. 
    Since $\|Tz_k\| \to \|T\|$, there exists a subsequence $(z_{\sigma(k)})\subseteq (z_k)$ such that $Tz_{\sigma(k)}\to \theta y_0$ for some $\theta\in \mathbb{T}$. So we would have $T z_{\sigma (k)} \in \theta U(y_0, n_0^{-1})$ for some $k\in \mathbb{N}$, a contradiction. Therefore, $\UQNA(X, Y) \subseteq \bigcap_{n\in \mathbb{N}} A_{n^{-1}}.$
    
    On the other hand, suppose that $T\in \bigcap_{n\in \mathbb{N}} A_{n^{-1}}$. Then, {for each $n\in \mathbb{N}$, there exist $y_n\in Y$ and $\delta_n>0$ such that}
    \begin{equation}\label{eq:tildeS}
        \widetilde{S}(T, \delta_n)\subseteq \mathbb {T} U \left(y_n, n^{-1} \right).
    \end{equation}
    We may assume that $\delta_n\to 0$ decreasingly, and noting that $\tilde{S}(T, \eta)=\mathbb{T}\tilde{S}(T, \eta)$, we can find from \eqref{eq:tildeS} a sequence $(x_n)\subseteq B_X$ such that 
    \begin{equation}\label{eq:limT}
        \|Tx_n\|>\|T\|-\delta_n \text{ and } \|Tx_n-y_n\|< n^{-1}
    \end{equation} for each $n\in \mathbb{N}$. 

    Now, we will prove that $(Tx_n)$ converges to an element of $\|T\|S_Y$. For $m>n$ we have $\|Tx_m-y_m\|\leq m^{-1}$ and 
    $$
    Tx_m\in \tilde{S}(T, \delta_m)\subseteq\tilde{S}(T, \delta_n)\subseteq \mathbb {T} U\left(y_n, n^{-1} \right).
    $$ Thus, there exists $\alpha_{m,n}\in \mathbb{T}$ such that $\|\alpha_{m, n}Tx_m-y_n\|< n^{-1}$ and therefore
    \begin{equation}\label{eq:ymn}
        \|\alpha_{m,n} y_m-y_n\|\leq \|\alpha_{m,n}y_m-\alpha_{m,n}Tx_m\|+\|\alpha_{m,n}Tx_m-y_n\| < \frac{1}{m}+\frac{1}{n}
    \end{equation}
    if $m>n$. Write $z_1=y_1$ and $z_k=\beta_k y_k$, where we define $$\beta_k=\alpha_{k,k-1}\cdots \alpha_{3,2}\alpha_{2,1} \in \mathbb{T}$$ for each $k\geq 2$. 
    Then, we have that
    \begin{equation*}
        {\|z_{k+1}-z_k\|=|\alpha_{k,k-1}\cdots\alpha_{2,1}|\|\alpha_{k+1,k}y_{k+1}-y_k\|< \frac{1}{k+1}+\frac{1}{k}}
    \end{equation*}
    by \eqref{eq:ymn}, so $(z_k)$ is a Cauchy sequence and hence $\lim_k z_k=\lim_k\beta_k y_k=z_0$ for some $z_0\in Y$. Passing to a subsequence we can assume that $(\beta_k)$ converges so that $y_k\to \theta z_0$ for some $\theta\in\mathbb{T}$.
    From \eqref{eq:limT} we have that $Tx_k\to \theta z_0$ and therefore
    \begin{equation*}
\|z_0\|=\lim\nolimits_k\|y_k\|=\lim\nolimits_k \|Tx_k\|=\|T\|,
    \end{equation*}
    so $\theta z_0\in\|T\|S_Y$.
    
    Finally, suppose that $(u_n)\subseteq B_X$ is a sequence satisfying $\|Tu_n\|\to \|T\|$. Then, for each $n\in \mathbb{N}$ there exists $\sigma(n)\in \mathbb{N}$ such that $Tu_{\sigma(n)}\in \tilde{S}(T, \delta_n)$ for every $n \in \mathbb{N}$. Thus, applying \eqref{eq:tildeS} there exists $\gamma_n\in\mathbb{T}$ such that 
    \begin{equation*}
        \|Tu_{\sigma(n)}-\gamma_n y_n\|< n^{-1}.
    \end{equation*}
    Passing again to a subsequence we may assume that $\gamma_n$ converges to $\gamma\in \mathbb{T}$. As $\lim_n y_n=\theta z_0$ we finally see that $(Tu_{\sigma(n)})$ converges to $\gamma \theta z_0$.
\end{proof}

Observe that the above proposition, combined with Proposition \ref{prop:una=na+uqna}, will enable us to obtain not only the denseness but also the residuality of $\RSE(X, Y)$ given that $\UQNA(X, Y)$ is dense and $\NA(X, Y)$ is residual in $\Lin (X,Y)$. As we will see in Example~\ref{example:beta+RNP-notimplyRSEdense}, the mere denseness of $\NA(X, Y)$ is not enough.

\begin{cor} 
    Let $X$ and $Y$ be Banach spaces. If $\UQNA(X, Y)$ is dense and $\NA(X, Y)$ is residual, then $\RSE(X, Y)$ is residual.
\end{cor}

It is worth mentioning that we do not know if the fact that $\NA(X,Y)$ is residual actually implies that $\ASE(X,Y)$ is dense, see (Q3) in \cite[\S 4]{JMR23}, and that this is the case when $X$ and $Y^*$ are separable \cite[Theorem 4.1]{JMR23}. 

\section{\texorpdfstring{$\RSE$}{RSE} version of Classical and Recent Results}\label{sec:classical-results}
Our goal in this section is to present $\RSE$ versions of the results from Acosta \cite{acosta1} on the failure of Lindenstrauss property B for strictly convex infinite-dimensional spaces, 
Uhl \cite{Uhl} on operators from $L_1(\mu)$ spaces, Schachermayer \cite{Schachermayer} on weakly compact operators from $C(K)$ spaces, and Johnson--Wolfe \cite{JW79} on compact operators from $C(K)$ and $L_1(\mu)$ spaces. In the later case, we actually generalize the results to finite-rank operators from Banach spaces with a good linear behaviour of their set of norm-attaining functionals as it was done in \cite{KLMW} for the usual norm attainment. Let us comment that, our results in this section apply to most classical Banach spaces as domain such as $C_0(L)$ spaces and   $L_p(\mu)$ spaces for $1\leq \mu\leq \infty$. In the case of $L_p(\mu)$ the result is immediate when $1<p<\infty$ as these spaces have the RNP, so we actually obtain that $\ASE(L_p(\mu),Y)$ is dense for every $Y$ \cite{Bourgain77}. 

\subsection{Nonexistence of a \texorpdfstring{$\RSE$}{RSE} Universal Infinite-Dimensional Range Space
}\label{subsec:Acosta-RSE}
Let us directly enunciate the result in the title.

\begin{theorem}\label{theorem:no-infdim-hasRSE-B}
Let $Y$ be an infinite-dimensional Banach space. Then, there is a Banach space $X$ such that $\RSE(X,Y)$ is not dense in $\Lin (X,Y)$. Moreover, such space $X$ can be chosen independently of the choice of $Y$.
\end{theorem}

Recall that Acosta \cite{acosta1} showed that no \emph{strictly convex} infinite-dimensional Banach space has Lindenstrauss' property B. Our proof actually follows the same general strategy as Acosta's. However, our goal is to demonstrate that the strict convexity assumption on the range space can be removed in this setting, when working with $\RSE$ operators instead of norm-attaining operators. In that paper, a Banach space was constructed as a universal domain in which the set of norm-attaining operators is not dense. For convenience, we present that space here along with the properties we will need.

\begin{example}[\mbox{\cite[Lemma~2.1, Lemma~2.2]{acosta1}}]\label{example:Acosta-X(w)}
Let $w=(w_n)\in \ell_2\backslash \ell_1$ be a decreasing sequence of positive scalars with $w_1<1$ and let $Z(w)$ be space of sequences $z$ of scalars satisfying that
    \begin{equation*}
        \|z\|:=\|(1-w)z\|_{\infty}+\|wz\|_1<\infty,
    \end{equation*}
which is a Banach space when endowed with the above formula as norm.
We denote by $(e_n)$ the sequence of coordinate functionals in $Z(w)^*$ given by
    \begin{equation*}
        e_n(z)=z(n) \quad \forall z\in Z(w),\ \forall n\in\N
    \end{equation*}
and define $X(w):=\overline{\spann}\{e_n\colon n\in \N\}\subseteq Z(w)^*$.
The following properties hold:
    \begin{itemize}
        \item[(i)] $(e_n)$ is a one-unconditional normalized basis of $X(w)$. 
        \item[(ii)] $X(w)^*\equiv Z(w)$ and $B_{X(w)}= \overline{\conv}(E)$,
        where 
        \begin{equation*}
            E=\left\{\theta_m(1-w_m)e_m+\sum_{i=1}^{n} \theta_iw_ie_i\colon m,n\in \N,\; \theta_i\in \mathbb{T},\, \forall i\right\}.
        \end{equation*}
        \item[(iii)] Given $x_0\in S_{X(w)}$ and $N\in\N$, there are $n\geq N$ and $\delta_n>0$ such that $\|x_0\pm\lambda e_n\|\leq 1$ for $\lambda\in \K$ so that $|\lambda|\leq \delta_n$.
    \end{itemize}
\end{example}

The crucial point in which we will use $\RSE$ operators in the proof of the theorem is isolated in the following easy lemma, which we will also use later on.

\begin{lemma}\label{lemma:easy-lemma-to-kill-propertyRSE-B}
Let $X$ and $Y$ be Banach spaces. If $S\in \RSE(X,Y)$ attains its norm at $x_0\in S_X$ and $\|x_0+z\|,\|x_0-z\|\leq 1$ for some $z\in X$, then $S(z)=0$.
\end{lemma}

\begin{proof}
Take $y_0^*\in S_{Y^*}$ such that $\re y_0^*(Sx_0)=\|Sx_0\|=\|S\|$. By Lemma~\ref{lem:strexp}, we have that $y_0^*$ strongly exposes $\overline{S(B_{X})}$ at $Sx_0$. Now, we have that
\begin{align*}
\|S\| &= \re y_0^* (Sx_0) = \re y_0^*\left(\frac{1}{2}(Sx_0+Sz)+\frac{1}{2} (Sx_0-Sz) \right) \\ 
& = \frac{1}{2}\re y_0^*\left(S(x_0+z)\right)+\frac{1}{2} \re y_0^*\left(S(x_0-z) \right) \\ &\leq 
\frac{1}{2} \|S\|\|x_0+z\| + \frac{1}{2} \|S\|\|x_0-z\|.
\end{align*}
As $y_0^*$ strongly exposes $\overline{S(B_{X})}$ at $Sx_0$, it follows that $S(x_0 \pm z)=Sx_0$, hence $S(z)=0$ as desired.
\end{proof}

\begin{proof}[Proof of Theorem~\ref{theorem:no-infdim-hasRSE-B}]
Consider a sequence $w\in\ell_2\backslash\ell_1$ satisfying the conditions required in the above example and let $X=X(w)$. By the Dvoretzky--Rogers' theorem (see \cite[Theorem 1.2]{DiestelJarchowTonge}), there exists a sequence $(y_n)\subseteq S_Y$ such that $\sum_{n=1}^\infty w_ny_n$ converges unconditionally, so the set 
\begin{equation*}
       A:=\left\{\sum_{n=1}^\infty \theta_nw_ny_n\colon \theta_n\in\K,\, |\theta_n|\leq 1,\, \forall n \right\} 
\end{equation*}
is bounded (use \cite[Theorem 1.9]{DiestelJarchowTonge}, for instance). Define $T\colon \spann\{e_n\colon n\in \N\}\subset X(w)\to Y $ by $T(e_n)=y_n$ for all $n\in \N$. Observe that 
\begin{equation*}
        \left\|T\left(\theta_m(1-w_m)e_m+\sum_{i=1}^{n} \theta_iw_ie_i\right)\right\|\leq 1+\left\|\sum_{i=1}^n \theta_iw_iy_i\right\|\leq 1 + \sup_{y\in A}\|y\|<\infty
\end{equation*}
for all $n,m\in \N$ and $\theta_i\in \mathbb{T}$. Hence $T$ is bounded on the set $E$ of Example~\ref{example:Acosta-X(w)} and, therefore, $T$ can be extended to a bounded linear operator from $X(w)$ to $Y$, which we will also denoted by $T$. Let us show that $\|T-S\|\geq 1$ whenever $S\in\RSE(X(w), Y)$, finishing thus the proof. Indeed, suppose that there exists $S \in \RSE(X(w),Y)$ which attains its norm at $x_0\in S_{X(w)}$. Use (iii) in Example~\ref{example:Acosta-X(w)} to find $n\in\N$ and $\delta_n>0$ such that
\begin{equation*}
        \|x_0 + \delta_n e_n\|,\|x_0 - \delta_n e_n\|\leq 1.
\end{equation*}
It follows from Lemma~\ref{lemma:easy-lemma-to-kill-propertyRSE-B} that $Se_n=0$. Since $\|e_n\|=1$, we have
\begin{equation*}
        \|T-S\|\geq \|Te_n-Se_n\|=\|y_n\|=1.\qedhere
\end{equation*}
\end{proof}

As a first consequence, we obtain that the fact that $\NA(X,Y)$ and $\UQNA(X,Y)$ are dense does not imply that $\RSE(X,Y)$ is, despite the fact that the latter is the intersection of the other two sets. Even more, the fact that $Y$ has property $\beta$ (which implies property B and so the denseness of $\NA(X,Y)$ for every $X$ \cite{Lin}) and the RNP (which implies that $\UQNA(X,Y)$ is residual for every $X$) is not enough to ensure the denseness of $\RSE$ operators. The following example illustrates this point:

\begin{example}\label{example:beta+RNP-notimplyRSEdense}
There is a Banach space $Y$ with Lindenstrauss' property $\beta$ and the RNP, and a Banach space $X$ such that $\RSE(X,Y)$ is not dense in $\Lin(X,Y)$.

Indeed, it suffices to consider an equivalent renorming of $\ell_2$ (say) with Lindenstrauss' property $\beta$ \cite{Partington}, and apply Theorem~\ref{theorem:no-infdim-hasRSE-B}.
\end{example}

We finish this subsection by noting that an analogue of Theorem~\ref{theorem:no-infdim-hasRSE-B} for domain spaces does not hold, since Banach spaces with the RNP satisfy the $\RSE$ version of property A \cite{Bourgain77}; that is, they serve as universal domain spaces for the denseness of $\RSE$ operators (actually, for $\ASE$ operators).

\subsection{Operators from \texorpdfstring{$L_1(\mu)$}{L1}-Spaces}\label{subsec:L1}
In the present subsection, we will restrict our focus to finite measures even though some results may hold without this restriction. This is because the discussion on the analytic version of the RNP is always done only for finite or even probability measures, for instance, in \cite{DU}. We actually refer to this book for background.

We start by proving that every strong RN operator from $L_1(\mu)$ to a Banach space $Y$ can be approximated by elements of $\RSE(L_1(\mu), Y)$.  This improves the classical result of Uhl \cite[Theorem 1]{Uhl}, which shows that operators from $L_1(\mu)$ to a Banach $Y$ with the RNP which attain their norm (in the classical sense) are dense. Our proof closely follows Uhl's original argument.  

\begin{theorem}\label{thm:L1}
Let $(\Omega,\Sigma, \mu)$ be a finite measure space, $Y$ a Banach space and $T\in \mathcal{L}(L_1(\mu), Y)$ be a strong RN operator. Then, for all $\varepsilon>0$ there exists $S\in \RSE(L_1 (\mu), Y)$ such that $\|T-S\|\leq \varepsilon$. Moreover, there is $\beta>0$ such that
$$
\overline{S(B_{L_1 (\mu)})}\subseteq \beta\overline{T(B_{L_1 (\mu)})}.
$$
\end{theorem}

We will use some well-known facts that we state here for commodity of the reader. 
Recall that an operator $T \in \Lin (L_1(\mu) , Y)$ is \emph{representable} if there exists an essentially bounded Bochner integrable function $g\colon \Omega \to Y$ such that 
    $$
    T(f)=\int_{\Omega} fg \,d\mu
    $$ 
for all $f\in L_1(\mu)$. 

\begin{lem}\mbox{\textup{(\cite[Lemma 5.4]{FJLPPQ} and \cite[Theorem III.1.5]{DU})}} \label{lem:representable}
    Let $(\Omega, \Sigma, \mu)$ be a finite measure space and $Y$ a Banach space. Then $T \in \Lin (L_1(\mu) , Y)$ is an RN operator if and only if $T$ is representable. Moreover, $Y$ has the RNP with respect to $(\Omega, \Sigma, \mu)$ if and only if every bounded linear operator from $L_1 (\mu)$ to $Y$ is representable. 
\end{lem}

We will also need the following geometrical lemma whose proof is straightforward. 

\begin{lem}\label{lem:2}
    Let $C$ be a subset of a Banach space $X$. If $x_0$ is a denting point of $ \overline{\conv}(C)$, then $x_0\in \overline{C}$. In particular, $\str(\overline{\conv}(C))\subseteq{\overline{C}}$.
\end{lem}
 
\begin{proof}[Proof of Theorem~\ref{thm:L1}]
    Let $T\in \mathcal{L}(L_1(\mu),Y)$ be a strong RN operator and $\varepsilon >0$. By Lemma \ref{lem:representable}, we can find an essentially bounded Bochner integrable function $g\colon \Omega \to Y$ such that 
    $$
    T(f)=\int_{\Omega} fg \, d\mu, \quad \forall f \in L_1 (\mu).
    $$   
    Then there exists a countably valued $\mu$-measurable function $h\colon\Omega \to Y$ such that $\esssup \|g-h\|_Y < \frac{\varepsilon}{6}$ and 
   \begin{equation*}
     h=\sum\limits_{i=1}^{\infty} z_i \chi_{E_i},
    \end{equation*}
    where $z_i\in Y$, $E_i\in \Sigma$, $\mu(E_i)>0$ and $E_i\cap E_j=\emptyset$ for $i\neq j$ (see the paragraph before \cite[Corollary II.1.3]{DU}). Consider 
    \[
    h_1:= \sum\limits_{i=1}^{\infty} y_i \chi_{E_i}, \,\text{ where }\, y_i : = \frac{1}{\mu(E_i)}\int_{E_i} g\,d\mu = T\left(\frac{\chi_{E_i}}{\mu(E_i)}\right), \, \forall i \in \mathbb{N}.
    \]
    Then, for every $i\in \mathbb{N}$ we have that
   \begin{equation*}
   \begin{split}
        \|h_1-h\|_{E_i} &=\left\|\frac{1}{\mu(E_i)}\int_{E_i} g d\mu -z_i\right\| =\left\|\frac{1}{\mu(E_i)}\int_{E_i} (g-h) d\mu\right\| 
    \leq \frac{\varepsilon}{6}.
   \end{split}   
   \end{equation*}
   This shows that $\esssup \|h_1-h\|_Y \leq\frac{\varepsilon}{6}$; therefore $\esssup \|g-h_1\|_Y <\frac{\varepsilon}{3}$. 
   
   Now, define $T_1\in \mathcal{L}(L_1(\mu), Y)$ by
   \begin{equation*}
       T_1(f)=\int_{\Omega} f h_1 \, d\mu, \quad \forall f \in L_1 (\mu).
   \end{equation*}
   Then, we have that $\|T_1-T\|<\frac{\varepsilon}{3}$ and $$\overline{T_1(B_{L_1 (\mu)})}=\overline{\aconv}\{y_n \colon n\in \mathbb{N}\}\subseteq \overline{T(B_{L_1 (\mu)})}.$$
   {Indeed, notice firstly that $y_n = T\left(\frac{\chi_{E_n}}{\mu(E_n)}\right) = T_1\left(\frac{\chi_{E_n}}{\mu(E_n)}\right)$ for each $n \in \mathbb{N}$. Secondly, observe that if $f\in B_{L_1(\mu)}$ one has }
    \begin{equation*}
        T_1(f)=\sum\limits_{i=1}^{\infty} y_i \int_{E_i}f d\mu 
    \end{equation*}
    and $\sum\limits_{i=1}^{\infty}\left|\int_{E_i} f\,d\mu \right| \leq \int_\Omega |f| d\mu\leq 1$, so we have $T_1(f)\in \overline{\aconv}\{y_n \colon n\in \mathbb{N}\}$. This, in particular, shows that $T_1$ is a strong RN operator.  
   
  On the other hand, we can assume that $T_1\neq 0$ (otherwise the proof would be trivial) and let $\beta=\sup \|y_i\| >0$. We can choose $i_0$ such that $\beta - \|y_{i_0}\|<\varepsilon/4$ and $\alpha >1$ such that $$\varepsilon/4 < (\alpha -1)\|y_{i_0}\|<\varepsilon/3.$$ 
  Observe that 
  \begin{equation}\label{eq:assumption_on_beta_alpha}
      \beta < \|y_{i_0}\| + \frac{\varepsilon}{4} < \alpha \|y_{i_0}\|.
  \end{equation}
  Next, define $T_2 \in \Lin (L_1 (\mu), Y)$ by
\begin{equation*}
    T_2(f) :=\int_{\Omega} fh_2 \, d\mu, \,\, \forall f\in L_1 (\mu), \, \text{ where } \,         h_2 := \sum_{i\neq i_0} y_i \chi_{E_i}+\alpha y_{i_0}\chi_{E_{i_0}}.
\end{equation*}
Since $(\alpha-1)\|y_{i_0} \| < \varepsilon/3$, we obtain that $\|T_1-T_2\|<\varepsilon/3$. Note also that $T_2(\chi_{E_{i_0}}/\mu(E_{i_0}))=\alpha y_{i_0}$. Arguing as before, observe that 
    \[
    \overline{T_2(B_{L_1 (\mu)})}=\overline{\aconv}(\{y_i \colon i\neq i_0 \}\cup \{\alpha y_{i_0}\}) \subseteq \alpha \,  \overline{T_1(B_{L_1 (\mu)})},
    \]
so $\|T_2\|=\alpha \|y_{i_0}\|=\|T_2(\chi_{E_{i_0}}/\mu(E_{i_0}))\|$, and $T_2$ is a strong RN operator.
    
By Lemma \ref{lem:strexp+density}, if we prove that $\alpha y_{i_0}$ is a strongly exposed point of $\overline{T_2(B_{L_1 (\mu)})}$ we will have that $T_2\in \overline{\RSE(L_1(\mu), Y)}$, which gives us an $S\in \RSE(L_1(\mu),Y)$ such that $\|T_2-S\|<\varepsilon/3$ and therefore $\|T-S\|<\varepsilon$, finishing the proof. 

Since $\overline{T_2(B_{L_1(\mu)})}$ is an RNP set, it is the closed convex hull of its strongly exposed points (see \cite[Theorem 3.5.4]{bourgin}). Therefore, there exists a strongly exposed point $z\in \overline{T_2(B_{L_1 (\mu)})}$ such that 
    \[
    \|z\|> \|y_{i_0}\| + \frac{\varepsilon}{4}.
    \] 
By Lemma \ref{lem:2}, the point $z$ belongs to the set $\overline{\mathbb{T}(\{y_i \colon i\neq i_0 \}\cup \{\alpha y_{i_0}\})}$. However, by the inequality \eqref{eq:assumption_on_beta_alpha}, the only possibility is $z\in  \mathbb{T}\{\alpha y_{i_0}\}$. In particular, $y_{i_0}$ is a strongly exposed point, as we wanted.

Finally, the moreover part follows since 
$$
\overline{T_2(B_{L_1 (\mu)})}\subseteq \alpha\overline{T_1(B_{L_1 (\mu)})} \subseteq \alpha \overline{T(B_{L_1 (\mu)})},
$$
and the operator $S$ is obtained from $T_2$ using Lemma~\ref{lem:strexp+density} and it is shown in Remark~\ref{remark:comment-to-lemma26} that $S$ can be taken in such a way that 
$$
\overline{S(B_{L_1 (\mu)})}\subseteq \alpha'\overline{T_2 (B_{L_1 (\mu)})}.
$$
Therefore, $\beta=\alpha\alpha'$ works.
\end{proof}

We would like to emphasize two consequences of Theorem~\ref{thm:L1}. The first one follows immediately since every operator whose range space has the RNP is obviously a strong RN operator; the second one follows since weakly compact operators are strong RN and we can use the moreover part of the theorem to get that the approximating operators can be taken to be weakly compact.

\begin{corollary}\label{corollary:domainL1-RNP}
Let $(\Omega,\Sigma, \mu)$ be a finite measure space and let $Y$ be a Banach space with the RNP. Then, $\RSE (L_1(\mu),Y)$ is dense in $\mathcal{L}(L_1(\mu),Y)$.
\end{corollary}

\begin{corollary}\label{corollary:domainL1-weaklycompact}
Let $(\Omega,\Sigma, \mu)$ be a finite measure space and let $Y$ be a Banach space. Then,
$$\mathcal{W}(L_1(\mu),Y)=\overline{\mathcal{W}(L_1(\mu),Y)\cap \RSE (L_1(\mu),Y)}.$$
\end{corollary}

The analogous result to this last one for compact operators also holds and can be proved in the same way, but we will also get it with a completely different argument in Corollary \ref{corollary--bigoneforfiniterank-compact-fordomainspaces}.

Our next aim is to show that, actually, Corollary \ref{corollary:domainL1-RNP} characterizes the RNP of the range space $Y$ when the measure $\mu$ is not purely atomic (of course, if the measure $\mu$ is purely atomic, then $L_1(\mu)$ has the RNP and, actually, $\ASE$ operators are dense for all $Y$'s).

\begin{theorem}\label{teo:uhl}
Let $(\Omega, \Sigma,\mu)$ be a finite measure space which is not purely atomic and let $Y$ be a Banach space. If $\RSE(L_1(\mu), Y)$ is dense in $\ELL(L_1(\mu), Y)$, then $Y$ has the RNP.
\end{theorem}

For the proof of Theroem \ref{teo:uhl}, we need the following preliminary result from measure theory which we highlight here for the reader's convenience. 

\begin{lem}[\mbox{Proof of \cite[Lemma III.1.4]{DU}}] \label{lem:bochner} 
If $(\Omega, \Sigma, \mu)$ is a finite measure space and $g\colon \Omega \to Y$ is a essentially bounded Bochner integrable function, then
\begin{equation*}
        T(f)=\int_{\Omega} fg \, d\mu \qquad \bigl(f\in L_1(\mu)\bigr)
\end{equation*}
defines a bounded linear operator from $L_1 (\mu)$ to $Y$ with $\|T\|=\essup \|g\|_Y$.
\end{lem}

The next preliminary result allows to pass the denseness of $\RSE$ operators from de domain space to an $L$-summand. The analogous results for other types of norm-attaining operators are well-known.

\begin{lem}\label{lemma:goingdown-Lsummand}
Let $X$, $Y$ be Banach spaces and suppose that $M$ is an $L$-summand of $X$ (that is, $X=M\oplus_1 N$ for some closed subspace $N$ of $X$). If $\RSE(X,Y)$ is dense in $\ELL(X,Y)$, then $\RSE(M,Y)$ is dense in $\ELL(M,Y)$.
\end{lem}

\begin{proof} Consider $T\in \ELL(M,Y)$ and define $\tilde{T}\in \ELL(X,Y)$ by $\tilde{T}(m,n)=T(n)$ for $(n,m)\in X=M\oplus_1 N$, which satisfies that $\|\tilde{T}\|=\|T\|$. For a given $\eps>0$ with $\eps<\|T\|$, we find $\tilde{S}\in \RSE(X,Y)$ witnessed by $(m_0,n_0)\in S_{X}$ such that $\|\tilde{T}-\tilde{S}\|<\eps$, and define $S\in \ELL(M,Y)$ by $S(m)=\tilde{S}(m,0)$ for $m\in M$ which clearly satisfies that $\|T-S\|<\eps$. Then, $\|S\|=\|\tilde{S}\|$. Indeed, the inequality $\|S\|\leq \|\tilde{S}\|$ is immediate and, conversely, 
$$
\|\tilde{S}\|=\max\{\|\tilde{S}\restricted_M\|,\|\tilde{S}\restricted_N\|\} \leq \max\{\|S\|,\|\tilde{T}\restricted_N\| + \|\tilde{T}-\tilde{S}\|\}=\|S\|.
$$
Now, if we have a sequence $(m_k)$ in $B_M$ such that $\|S(m_k)\|\to \|S\|$, then $$\|\tilde{S}(S)(m_k,0)\|\to \|\tilde{S}\|$$ and, therefore, there is a subsequence $(m_{\sigma(k)})$ and $\theta\in \mathbb{T}$ such that $$\tilde{S}(m_{\sigma(k)},0)\to \theta \tilde{S}(m_0,n_0).$$ In particular, $S(m_{\sigma(k)})\to \theta S(m_0)$. This shows that $S\in \RSE(M,Y)$, finishing the proof.
\end{proof}

The next result include the technical details that we need to use in the proof of Theorem~\ref{teo:uhl}. It provides the  ``$\RSE$ version" of \cite[Lemma 2]{Uhl}, but there is no need to assume that the range space is strictly convex.

\begin{lem}\label{lem:rse}
Let $(\Omega, \Sigma, \mu)$ be a finite measure space and let $Y$ be a Banach space. If $T\in\RSE(L_1 (\mu), Y)$, then there exists a set $E_0\in \Sigma$ with $\mu(E_0)>0$, $g\in L_{\infty}(\mu)$ with $|g|=1$ on $E_0$ and $y_0\in Y$ with $\|y_0\|=\|T\|$ such that
    \begin{equation*}
        T(f\chi_{E_0})=\left(\int_{E_0} fg \, d\mu\right) y_0, \, \text{ } \forall f\in L_1(\mu).
    \end{equation*}
\end{lem} 

\begin{proof}
If $\|T\|=0$ there is nothing to prove. Otherwise, let $f_0 \in S_{L_1(\mu)}$ such that $\|Tf_0 \|=\|T\|$ and $y_0^*\in S_{Y^*}$ so that
    \begin{equation*}
        y_0^*(Tf_0)=\|T\|.
    \end{equation*}
    Then, $y_0^*$ strongly exposes $\overline{T(B_{L_1(\mu)})}$ at $Tf_0$. Choose $h\in L_{\infty}(\mu)$ so that $\|h\|_\infty = \|T\|$ and
    \begin{equation*}
        y_0^*(Tf)=\int_{\Omega} fh \, d\mu
    \end{equation*}
    for all $f\in L_1(\mu)$. Let $E_0$ be the support of $f_0$ ($\mu(E_0)>0)$ as $\|f_0\|>0$), then it must be $h=\overline{\sgn} f_0 \|T\|$ on $E_0$, where $\overline{\sgn}f_0(t)=\frac{\overline{f_0}(t)}{|f_0(t)|}$ for $t\in E_0$ and $\overline{\sgn}f_0(t)=0$ elsewhere. Therefore, for $f\in L_1(\mu)$
    \begin{equation*}
        y_0^*(T(f\chi_{E_0}))=\int_{E_0} f \overline{\sgn} f_0 \|T\| \, d\mu.
    \end{equation*}

    Next, suppose that $E\subseteq E_0, E\in \Sigma$ and $\mu(E)>0$. Then 
    \begin{equation*}
        y_0^*T\left(\frac{\chi_E}{\mu(E)}\sgn f_0\right)=\int_{E_0} \frac{\chi_E}{\mu(E)}\sgn f_0\overline{\sgn} f_0 \|T\|d\mu =\|T\|
    \end{equation*}
    and
    \begin{equation*}
        y_0^*T\left(\frac{\chi_{E_0}}{\mu(E_0)}\sgn f_0\right)=\int_{E_0} \frac{\chi_{E_0}}{\mu(E_0)}\sgn f_0\overline{\sgn} f_0 \|T\|d\mu =\|T\|.
    \end{equation*}
    Hence, as $y_0^*$ strongly exposes $\overline{T(B_{L_1(\mu)})}$,  it must be
    \begin{equation}\label{eq:subset_of_E}
        \frac{T\left(\chi_E\sgn f_0\right)}{\mu(E)}=\frac{T\left(\chi_{E_0}\sgn f_0\right)}{\mu(E_0)}=Tf_0:=y_0.
    \end{equation}
    Now, let $f\in L_1(\mu)$ be a simple function, $\eps>0$ and $\varphi\in  L_1(\mu)$ another simple function so that $\|\overline{\sgn} f_0-\varphi\|_{\infty}<\eps$. We have that $T(f\chi_{E_0})=T(f \overline{\sgn} f_0 \sgn f_0)$ and therefore
    \begin{equation*}
    \begin{split}
        \|T(f\chi_{E_0})-T(f\varphi\sgn f_0 \chi_{E_0})\|&\leq \|T\|\|f\overline{\sgn} f_0 \sgn f_0-f\varphi\sgn f_0 \chi_{E_0}\|\\
       & =  \|T\|\int_{E_0} |f| |\overline{\sgn} f_0-\varphi| \, d\mu  < \|T\|\|f\| \eps.
    \end{split}
    \end{equation*}
    Choose $A_1, \dots, A_n\in\Sigma$ disjoint so that 
    \begin{equation*}
        f=\sum_{i=1}^n\alpha_i\chi_{A_i} \text{ and } \varphi=\sum_{i=1}^n\beta_i\chi_{A_i}.
    \end{equation*}
    Then
    \begin{equation*}
    \begin{split}
        T(f\varphi\sgn f_0 \chi_{E_0})&=\sum_{i=1}^n \alpha_i\beta_iT(\chi_{A_i\cap E_0} \sgn f_0)\\
        &=\sum_{\mu (A_i\cap E_0)>0} \alpha_i\beta_i\frac{T(\chi_{A_i\cap E_0} \sgn f_0)}{\mu (A_i\cap E_0)}\mu (A_i\cap E_0)\\
        &=\sum_{\mu (A_i\cap E_0)>0} \alpha_i\beta_i\mu (A_i\cap E_0) y_0 \quad {(\text{by \eqref{eq:subset_of_E}})}\\ 
        &=\left(\int_{E_0}f\varphi d\mu\right)y_0.
    \end{split}
    \end{equation*}
    Hence, for every $\eps>0$ 
    \begin{align*}
            &\left\|T(f\chi_{E_0}) - \left(\int_{E_0}f\overline{\sgn} f_0 d\mu\right) y_0 \right \| \\
            &\leq \left\|T(f\chi_{E_0})-\left(\int_{E_0}f\varphi d\mu\right)y_0 \right\|+ \left\|\left(\int_{E_0}f\varphi d\mu\right)y_0-\left(\int_{E_0}f\overline{\sgn} f_0 d\mu\right)y_0 \right\|\\
            &< \|T\|\|f\|\eps+ \|y_0\|\left| \int_{E_0}f(\overline{\sgn} f_0 -\varphi) \, d\mu
            \right| \leq 2\|T\|\|f\|\eps,
    \end{align*}
    so $$T(f\chi_{E_0})=\left(\int_{E_0}f\overline{\sgn} f_0 \, d\mu\right)y_0.$$
    As simple functions are dense, we obtain the equality for every $f\in L_1(\mu)$. 
\end{proof} 

We are now able to present the pending proof.

\begin{proof}[Proof of Theorem \ref{teo:uhl}]
    Let $T\in\ELL(L_1(\mu), Y)$ and $\eps>0$. Define a class $\mathcal{M}$ of measurable sets in $\Sigma$ by the condition that $E\in\mathcal{M}$ if there exists an essentially bounded Bochner integrable function $g(E, \eps)\colon \Omega\to Y$ such that 
    \begin{equation*}
        \left\|T(f\chi_E)-\int_{E}fg \,d\mu\right\|\leq \eps \|f\chi_E\|_1
    \end{equation*}
    for every $f \in L_1 (\mu)$.
    Note that $\mathcal{M}$ is closed by taking subsets that are measurable. Indeed, if $E \in \mathcal{M}$ and $A \subseteq E$ is measurable, then given $f \in L_1 (\mu)$, 
    \begin{align*}
        \left\|T(f\chi_A)-\int_{A}fg(E, \eps)d\mu\right\| &= \left\|T(f\chi_A\chi_E)-\int_{E}f\chi_A g(E,\varepsilon) d\mu\right\| \\ 
        &\leq \eps \|f\chi_A\chi_E\|_1 =\eps\|f\chi_A\|_1.
    \end{align*}
    Let $\alpha=\sup\{\mu(E)\colon E\in\mathcal{M}\}$ and $(E_n)\subseteq \mathcal{M}$ a sequence such that $\lim_n \mu(E_n)=\alpha$. Write $A_1=E_1, A_2=E_2 \setminus E_1, \ldots , A_n=E_n \setminus 
 \cup_{i=1}^{n-1} E_i$. Then $A_i\cap A_j=\emptyset$ if $i\neq j$, $\cup_{n=1}^{\infty} A_n=\cup_{n=1}^{\infty} E_n$ and $\mu\left(\cup_{n=1}^{\infty} A_n\right)\geq \mu(E_n)$ for all $n\in \N$, so $\mu\left(\cup_{n=1}^{\infty} A_n\right)\geq\alpha$. As $A_n\subseteq E_n\in \mathcal{M}$ there exists a sequence $(g_n)$ of essentially bounded Bochner integrable functions $g_n\colon\Omega \to Y$ such that for every $f\in L_1(\mu)$
    \begin{equation*}
        \left\|T(f\chi_{A_n})-\int_{A_n}fg_nd\mu\right\|\leq \eps \|f\chi_{A_n}\|_1.
    \end{equation*}
    Then,
    \begin{equation*}
        \left\|\int_{A_n}fg_n \, d\mu\right\|\leq \|T(f\chi_{A_n})\|+\eps \|f\chi_{A_n}\|_1\leq (\|T\|+\eps)\|f\|_1.
    \end{equation*}
    By Lemma \ref{lem:bochner} 
    \begin{equation*}
        \essup \|g_n\chi_{A_n}\|_Y=\sup_{\|f\|_1\leq 1} \left\|\int_{A_n}fg_n \, d\mu\right\|\leq \|T\|+\eps
    \end{equation*}
    and therefore $\sup_n \essup \|g_n\chi_{A_n}\|_Y \leq \|T\|+\eps$ Defining $g\colon\Omega\to Y$ by
    \begin{equation*}
        g(t)=\left\{ \begin{array}{lcc} g_n(t) & \text{if} & t\in A_n \\ \\ 0 & \text{if} & t\notin \cup_{n=1}^{\infty} A_n \end{array} \right.
    \end{equation*}
    we have that $\essup \|g\|_Y\leq \|T\|+\eps$ and if $f \in L_1(\mu)$, then 
    \begin{equation*}
        \begin{split}
            \left\|T(f\chi_{\cup_{n=1}^{\infty}A_n})-\int_{\cup_{n=1}^{\infty}A_n}fg \, d\mu\right\|&\leq\sum_{n=1}^{\infty} \left\|T(f\chi_{A_n})-\int_{A_n}fg_n \, d\mu\right\|\\
            &\leq \sum_{n=1}^{\infty} \eps\|f\chi_{A_n}\|=\eps \|f\chi_{\cup_{n=1}^{\infty}A_n}\|_1 \leq \eps\|f\|_1
        \end{split}
    \end{equation*}
    Therefore, $\cup_{n=1}^{\infty}A_n \in \mathcal{M}$, and so $\mu(\cup_{n=1}^{\infty}A_n)=\alpha$. 
    
Next, we will show that $\alpha= \mu(\Omega)$. Suppose, on the contrary, that $\alpha< \mu(\Omega)$. Set $B_0=\Omega \setminus \cup_{n=1}^{\infty}A_n$ and consider $L_1(B_0)$, the space of integrable functions in $L_1 (\mu)$ supported on $B_0$, that is, 
    \[
    L_1 (B_0) = \{ f \in L_1 (\mu) \colon f =0 \text{ a.e. on } \Omega\setminus B_0\}.
    \]
    Then, $L_1 (B_0)$ is an $L$-summand in $L_1 (\mu)$ via the band projection $f \mapsto f \chi_{B_0}$.  Define $T_1\colon L_1(B_0)\to Y$ by $T_1 := T\restricted_{L_1 (B_0)}$. By Lemma~\ref{lemma:goingdown-Lsummand}, we have that $\RSE (L_1 (B_0), Y)$ is dense in $\mathcal{L}(L_1 (B_0), Y)$. Thus, there exists $T_2\in \RSE(L_1(B_0), Y)$ satisfying that $\|T_1-T_2\|\leq \eps$. By Lemma \ref{lem:rse}, there exist $y_1\in Y$, $B_1\subseteq B_0$, $\mu(B_1)>0$ and $h \in L_{\infty}(\mu)$ with $|h|=1$ on $B_1$ such that
    \begin{equation*}
        T_2(f\chi_{B_1})= \left( \int_{B_1}fh \,d\mu\right) y_1, \,  \forall f\in L_1(B_0).
    \end{equation*}
   Writing $g'=(h\chi_{B_1}) y_1\colon \Omega \to Y$, we have
    \begin{equation*}
    \begin{split}
        \left\|T(f\chi_{B_1})-\int_{B_1}fg'd\mu\right\|&\leq \|T_1(f\chi_{B_1})-T_2(f\chi_{B_1})\|\\
        &\leq \|T_1-T_2\|\|f\chi_{B_1}\|\leq \eps\|f\chi_{B_1}\|,
    \end{split}
    \end{equation*}
    so $B_1\in\mathcal{M}$. Now, set $\tilde{g}=g+g'$, then
    \begin{align*}
            &\left\|T(f\chi_{\cup_{n=1}^{\infty} A_n \cup B_1})-\int_{\cup_{n=1}^{\infty} A_n \cup B_1}f\tilde{g} \, d\mu\right\| \\
            &\leq \sum_{n=1}^{\infty}  \left \|T(f\chi_{ A_n })-\int_{A_n } fg_n \, d\mu \right \|+ \left \|T(f\chi_{ B_1 })-\int_{B_1 } fg' \, d\mu \right \|\\
            &\leq \eps \sum_{n=1}^{\infty} \|f\chi_{A_n}\|+\eps\|f\chi_{B_1}\|=\eps\|f\chi_{\cup_{n=1}^{\infty} A_n \cup B_1}\|,
    \end{align*}
    so $\bigcup_{n=1}^{\infty} A_n \cup B_1 \in \mathcal{M}$ but
    \begin{equation*}
        \mu\left(\bigcup\nolimits_{n=1}^{\infty} A_n \cup B_1\right)=\mu\left(\bigcup\nolimits_{n=1}^{\infty} A_n\right)+\mu(B_1)>\alpha,
    \end{equation*}
    which is a contradiction. Therefore, $\alpha = \mu (\Omega) = \mu (\cup_{n=1}^\infty A_n)$ and
    \begin{equation*}
        \left\|T(f)-\int_{\Omega}f g \, d\mu \right\|\leq \eps\|f\|_1
    \end{equation*}
    for every $f\in L_1 (\mu)$. Finally, let $g_n\colon\Omega \to Y$ be a sequence of Bochner integrable essentially bounded functions such that for all $f\in L_1[0,1]$
    \begin{equation*}
        \left\|T(f)-\int_{\Omega}fg_n \, d\mu \right\|\leq \frac{1}{n}\|f\|_1
    \end{equation*}
    for all $n\in \N$. Then, by Lemma \ref{lem:bochner}, $$\essup\|g_n-g_m\|=\sup_{\|f\|\leq1} \left\|\int_{[0,1]} f(g_n-g_m)d\mu\right\|\to 0,$$ so $(g_n)$ is Cauchy and there exists a Bochner integrable essentially bounded function $g\colon\Omega\to Y$ such that $\lim_n\essup \|g_n-g\|_Y=0$. By the Dominated Convergence Theorem (see \cite[Theorem II.2.3]{DU}),
    \begin{equation*}
        \int fg \, d\mu=\lim_n \int fg_n \, d\mu=T(f)
    \end{equation*}
    for all $f\in L_1 (\mu)$. 
    
    We have proved so far that every operator from $L_1(\mu)$ to $Y$ is representable. Thus, Lemma \ref{lem:representable} shows that $Y$ has the RNP with respect to $(\Omega,\Sigma,\mu)$. Finally, since $\mu$ is not purely atomic, it follows from \cite[Theorem~2]{Chatterji} that $Y$ has the RNP.
\end{proof}

The following nice consequence of the results of this subsection follows easily. 

\begin{corollary}
Let $\mu$ and $\nu$ two finite measures. Then, $\RSE(L_1(\mu),L_1(\nu))$ is dense in $\ELL(L_1(\mu),L_1(\nu))$ if and only if one of the two measures is purely atomic.
\end{corollary}

\begin{proof}
If $\mu$ is purely atomic, then $L_1(\mu)$ has the RNP and therefore the denseness of $\RSE(L_1(\mu),L_1(\nu))$ follows from Bourgain's result \cite{Bourgain77}. If $\nu$ is purely atomic, then $L_1(\nu)$ has the RNP and we may use Theorem \ref{thm:L1}. For the converse, if $\mu$ is not purely atomic, it follows from Theorem \ref{teo:uhl} that $L_1(\nu)$ has the RNP. That is, $\nu$ must be purely atomic.
\end{proof}

Compare it with the result by Iwanik \cite{Iwanik} that $\NA(L_1(\mu),L_1(\nu))$ is dense in $\ELL(L_1(\mu),L_1(\nu))$ for arbitrary measures $\mu$ and $\nu$ \cite{Iwanik}. Actually, the above corollary can be considered as the $\RSE$ version of this result.

\subsection{Weakly Compact Operators from \texorpdfstring{$C(K)$}{C(K)}-Spaces}\label{subsec:W-on-CK}

Another significant result on norm-attaining operators on classical spaces was obtained by Schacher\-may\-er \cite{Schachermayer}. In his paper, he proved that for every compact Hausdorff space $K$ and every Banach space $Y$, $\mathcal{W}(C(K), Y)$ is equal to the norm-closure of the set $\mathcal{W}(C(K), Y) \cap \NA(C(K), Y)$. The following theorem improves this result by replacing $\NA$ with $\RSE$.

\begin{theorem}
Let $K$ be a compact Hausdorff space and $Y$ be a Banach space.
Then, $\mathcal{W}(C(K),Y) \cap \RSE(C(K),Y)$ is dense in $\mathcal{W}(C(K),Y)$.
\end{theorem}

\begin{proof}
We follow the argument of the proof of \cite[Theorem B]{Schachermayer}. Let $S \in \mathcal{W}(C(K),Y)$ be such that $\|S\|=1$ and let $\varepsilon>0$ be given. Since $S \in \mathcal{W}(C(K),Y)$, it is a strong RN operator. By \cite[Remark 3.4]{CCJM}, there exists $T \in \mathcal{W}(C(K), Y)$ such that 
\begin{enumerate}[label=(\roman*)]
\itemsep0.25em
    \item $T \in \UQNA(C(K),Y)$ towards some $y_0$; 
    \item $\|S-T\| < \varepsilon$; 
    \item \label{eq:y_0_strongly_exposed} $T^*$ attains its norm at some $y_0^* \in S_{Y^*}$ which strongly exposes $\overline{T(B_{C(K)})}$ at $y_0$ with $\sup \re y_0^* (T(B_{C(K)})) = 1$. 
\end{enumerate}   
Put $\mu_0 := T^* (y_0^*)$. Applying \cite[Lemma 3.1]{Schachermayer} to the weakly compact subset $T^* (B_{Y^*})$, there exists a bounded linear operator $S \colon C(K)^* \to C(K)^*$ with $\|S\|=1$ such that 
\begin{enumerate}[label=(\roman*')]
\itemsep0.25em
    \item $\| S \mu - \mu \| \leq \varepsilon$ for every $\mu \in T^* (B_{Y^*})$; 
    \item there exists $f_0 \in {C(K)}$ such that $\|f_0\|=1$ and $\langle f_0,S \mu_0 \rangle =\|\mu_0 \|=1$. 
\end{enumerate} 
Let $A={(ST^*)^*} \restricted_{C(K)}=T^{**}S^* \restricted_{C(K)} \colon C(K) \to Y$ (here, $A(C(K))\subset Y$ by weak-compactness of $T$). First, note that $A^*=ST^*$ and
\[ \|A-T\|=\|A^*-T^*\| = \|ST^*-T^*\| \leq \varepsilon. 
\]
By the part (b) of Lemma \ref{lem:strexp+density}, if we show that $A \in \NA(C(K),Y)$ at $f_0$ and that $Af_0$ is a strongly exposed point of $\overline{A(B_{C(K)})}$, then there exists $Q \in \RSE ( C(K), Y)$ such that $\|A-Q\| < \varepsilon$ and $A-Q$ is a rank-one operator. It follows that $Q \in \mathcal{W}(C(K),Y)\cap\RSE(C(K),Y)$ and $\|Q-S\| \leq \|Q-A\|+\|A-T\|+\|T-S\| < 3\varepsilon$.

To this end, observe that $\|A\| \leq 1$ and 
\begin{equation*}\label{eq:A_na}
 \|Af_0\| \geq y_0^*(Af_0) = (A^*y_0^*) (f_0) = (ST^*y_0^*)(f_0)= \langle f_0, S\mu_0 \rangle =1. 
 \end{equation*} 
Moreover, 
\begin{align*}
    1= y_0^* (Af_0) &\leq \sup \re y_0^* \bigl(\overline{T^{**} (B_{C(K)^{**}})} \bigr) \quad \bigl(\text{since } A(B_{C(K)}) \subseteq \overline{ T^{**} (B_{C(K)^{**}} } )  \bigr)  \\
    &=  \sup \re y_0^* \bigl(\overline{T (B_{C(K)} )} \bigr) \quad \bigl(\text{since } \overline{ T^{**} (B_{C(K)^{**}} } ) =  \overline{ T (B_{C(K)}) } \bigr) \\
    &= \re y_0^* (y_0) = 1  \quad (\text{due to \ref{eq:y_0_strongly_exposed}}) 
\end{align*} 
As $y_0^*$ strongly exposes $\overline{T(B_{C(K)})}$ at $y_0$, it follows that $y_0 = Af_0$ and that $y_0^*$ strongly exposes $\overline{A(B_{C(K)})}$ at $Af_0$, as desired.
\end{proof}

\subsection{Finite-Rank and Compact Operators}\label{subsec:finiterank-compact}
Let us start recalling that there are compact operators which cannot be approximated by norm attaining ones \cite{MarJFA20214} and that the existence of finite-rank operators of the same kind is an old open problem. We send the interested reader to \cites{Martin-RACSAM2016,KLMW} for more information. Therefore, our proofs of the denseness of $\RSE$ compact operators of $\RSE$ finite-rank operators require additional conditions to hold. Even more, in the case of compact operators, with the same ideas as those used in Subsection~\ref{subsec:Acosta-RSE}, we may prove the following result which is the $\RSE$ version of \cite[Proposition 6]{MarJFA20214}, where the strict convexity of the space $Y$ was required. Recall also that every Banach space can be renormed with Lindenstrauss' property $\beta$ \cite{Partington} and that this property implies that $\NA(X,Y)\cap \comp(X,Y)$ is dense in $\comp(X,Y)$ for every Banach space $X$ (see \cite[Proposition 4.1]{Martin-RACSAM2016}).

\begin{prop}\label{prop:no-universal-RSE-B-compact}
For every infinite-dimensional Banach space $Y$ failing the approximation property, there is a Banach space $X$ and a compact linear operator from $X$ to $Y$ which cannot be approximated by $\RSE$ operators.
\end{prop}

\begin{proof}
As $Y$ fails the approximation property, a result of Grothendieck combined with Enflo's example (cf.\ \cite[Theorem 1.e.5 and Theorem 2.d.6]{LT}), ensures the existence of a closed subspace $X$ of $c_0$ and a compact operator $T\colon X\to Y$ which cannot be approximated by finite-rank operators (see \cite[Theorem 18.3.2]{Jarchow}). To complete the proof, it is enough to prove that $\RSE(X,Y)\subseteq \FR(X,Y)$. Indeed, if $S\in \RSE(X,Y)$ attains its norm at $x_0\in S_X$, being $X$ a closed subspace of $c_0$, it follows that $Z:=\overline{\spann}\{z\in X\colon \|x_0\pm z\|\leq 1\}$ is of finite-codimension (see \cite[Lemma~2.2]{Martin-RACSAM2016} for the easy proof). By Lemma~\ref{lemma:easy-lemma-to-kill-propertyRSE-B}, it follows that $S(Z)=0$, so $S$ has finite-rank.
\end{proof}

Regarding the denseness of finite-rank norm-attaining operators, one of the main consequences in \cite{KLMW} establishes that if a Banach space $X$ satisfies that $\NA(X)$ contains sufficiently many linear subspaces, then the set $\FR \cap \NA$ of finite-rank norm-attaining operators on $X$ is dense in the ideal $\FR$ of finite-rank operators. In the next theorem, we improve upon these results by replacing $\FR\cap\NA$ with $\FR \cap \RSE$. The proof is closely related to the one of \cite[Theorem 3.3]{KLMW}, using the technical results we have presented in Section \ref{sec:pre}.

\begin{theorem} \label{theo-finite rank}
Let $X$ and $Y$ be Banach spaces. Suppose that for every $n \in \N$, every $\eps >0$ and all $x_1^*,\ldots, x_n^* \in B_{X^*}$, there are $y_1^*, \ldots, y_n^* \in B_{X^*}$ such that $\|x_i^*-y_i^* \| < \eps$, and $\spann\{y_1^*,\ldots, y_n^*\} \subseteq \NA(X)$. Then,
\[
\FR(X,Y) \subseteq \overline{\FR(X,Y) \cap \RSE(X,Y)}.
\]
\end{theorem} 

\begin{proof}
Let $T \in \FR(X,Y)$ and write $T=\sum_{i=1}^n x_i^* \otimes y_i$ for some $x_i^* \in X^*$ and $y_i \in Y$, $i=1,\ldots, n$. Given $\eps >0$, by assumption, we can find $z_1^*, \ldots, z_n^* \in X^*$ such that $\spann\{z_1^*,\ldots, z_n^* \} \subseteq \NA(X)$ and $R:= \sum_{i=1}^n z_i^* \otimes y_i$ satisfies that $\|R-T\| < \eps$. Then $(\ker R)^{\perp} \subseteq \spann\{z_1^*,\ldots, z_n^*\} \subseteq \NA (X)$.
Since $X/\ker R$ is finite-dimensional, we can find $\tilde{S} \in \ASE(X/\ker R, Y)$ such that $\|\tilde{S}- \tilde{R} \| < \eps$. Let $S:= \tilde{S} \circ q \in \Lin (X,Y)$, where $q\colon X \rightarrow X/\ker R$ is the canonical quotient map. Since $(\ker {S})^{\perp} \subseteq \NA (X)$ and $S$ is of finite-rank, we have that $\|Sx_0\| = \|(\tilde{S}\circ q) (x_0)\|= \|S\|$ for some $x_0 \in S_X$ (see \cite[Corollary 2.6]{KLMW}). Finally, Proposition \ref{prop:quotient_implication2} shows that $S \in \RSE(X,Y)$, finishing the proof.
\end{proof}

The following main consequence follows directly.

\begin{corollary}
Let $X$ be a Banach space such that there is a norm dense linear subspace of $X^*$ contained in $\NA(X)$. Then, for every Banach space $Y$, we have:
 \[ \FR(X,Y) \subseteq \overline{\FR(X,Y) \cap \RSE(X,Y)}.\] 
\end{corollary}

We conclude this section with a brief discussion on the density of $\FR(X,Y) \cap \RSE(X,Y)$ in $\mathcal{K}(X,Y)$. First, notice that the denseness of $\FR(X,Y) \cap \RSE(X,Y)$ in $\FR(X,Y)$ implies that $\comp(X,Y)$ coincides with $\overline{\FR(X,Y) \cap \RSE(X,Y)}$ whenever $Y$ or $X^*$ has the approximation property. 

\begin{corollary}
Let $X$ be a Banach space such that there is a norm dense linear subspace of $X^*$ contained in $\NA(X)$ and let $Y$ be a Banach space. If $X^*$ or $Y$ has the approximation property, then 
$$
\comp(X,Y) = \overline{\FR(X,Y) \cap \RSE(X,Y)}.
$$
\end{corollary}

As Theorem \ref{theo-finite rank} and its two corollaries strengthen the result of \cite[Theorem 3.3]{KLMW}, some consequences derived from it can also be improved, including the classical result by Johnson--Wolfe \cite{JW79} on compact operator acting from $C(K)$ and $L_1(\mu)$ spaces. Nevertheless, we provide a non-exhaustive list of such examples and we refer the reader to \cite{KLMW} for further details, references, and additional examples.

\begin{corollary}\label{corollary--bigoneforfiniterank-compact-fordomainspaces}
Let $X$ and $Y$ be Banach spaces. Then 
$$
\FR(X,Y) \subseteq \overline{\FR(X,Y) \cap \RSE(X,Y)}
$$
in the following cases: 
\begin{enumerate}
\itemsep0.25em
    \item $X$ satisfies that for all $x_1^*, \dots, x_n^* \in B_{X^*}$ and every $\varepsilon > 0$, there is a norm-one projection $P \in \Lin(X, X)$ of finite-rank such
that $\max_i \|x_i^*-P^*(x_i^*)\|<\varepsilon$. 
    \begin{enumerate}[label=\textup{(a\arabic*)}, itemsep=0.25em, before=\vspace{-1em}]
        \itemsep0.25em 
        \item $X= C_0 (L)$ where $L$ is a locally compact Hausdorff topological space. 
        \item $X= L_p (\mu)$ where $1\leq p \leq \infty$ and $\mu$ is a positive measure. 
    \end{enumerate}
\item $X$ is an isometric predual of $\ell_1$. 
\item $X$ has a shrinking monotone Schauder basis.
\item $X$ is a closed subspace of $c_0$ with monotone Schauder basis.
\item $X$ is a finite-codimensional proximinal subspace of $c_0$.
\item $X$ is a finite-codimensional proximinal subspace of $\mathcal{K}(\ell_2)$.
\item $X$ is a $c_0$-sum of reflexive spaces. 
\end{enumerate}
Moreover, in items (a), (b), (c), and (d), we actually have that 
$$
\comp(X,Y) = \overline{\FR(X,Y) \cap \RSE(X,Y)}.
$$
\end{corollary}

\section{Range Spaces with Quasi-ACK Structure}\label{sec:ACK}

In this section we will provide denseness results of $\RSE$ operator in the case of specific target Banach spaces $Y$---namely, those with the so-called quasi-ACK structure---under the hypothesis of the denseness of $\NA(X) \cap \smo({X^*})$  in $X^*$.

The notion of quasi-ACK structure was introduced in \cite{CJ23} as a generalization of ACK$_\rho$ structure \cite{CGKS2018} and property quasi-$\beta$ \cite{AAP2}. A Banach space $Y$ with a quasi-ACK structure serves as a universal range space, ensuring the denseness of norm-attaining operators from $X$ into $Y$ within a specific operator ideal class—namely, the $\Gamma$-flat operator ideal—for any arbitrary domain space $X$.

Let us present a formal definition of quasi-$\ACK$ structure. A Banach space $X$ is said to have \emph{quasi-$\ACK$ structure} if there exist a $1$-norming set $\Gamma \subseteq B_{X^*}$ and a function $\rho\colon \Gamma \to [0,1)$ such that for any $e^* \in \operatorname{ext}(B_{X^*})$, there exists $\Gamma_{e^*} \subseteq \Gamma$ satisfying that
\begin{enumerate}[label=(\roman*)]
\setlength\itemsep{0.25em}
\item \label{l1} $e^* \in \overline{\mathbb{T} \Gamma_{e^*}}^{w^*}$,
\item \label{l2} $\sup_{x^* \in \Gamma_{e^*}} \rho(x^*) <1$,
\item \label{l3} for every $\eps>0$ and a nonempty relatively $w^*$-open subset $U$ of $\Gamma_{e^*}$, there are a nonempty subset $V \subseteq U$, $x_1^* \in V$, $e \in S_X$, $F \in \Lin(X)$ such that 

\begin{enumerate}[label=(\roman*')]
\setlength\itemsep{0.25em}
\item \label{l1'} $\|F(e)\|=\|F\|=1$,
\item \label{l2'} $x_1^*(F(e))=1$,
\item  \label{l3'}$F^*(x_1^*)=x_1^*$,
\item \label{l4'} If we let $V_1 := \{ x^* \in \Gamma \colon \|F^*(x^*)\| + (1-\eps) \|x^* - F^*(x^*)\| \leq 1 \}$, then $|x^*(F(e))| \leq \rho(x_1^*)$ for any $x^* \in \Gamma \setminus V_1$,
\item  \label{l5'}$\dist (F^*(x^*), \operatorname{aco}\{0,V\}) < \eps$ for all $x^* \in \Gamma$,
\item  \label{l6'}$|v^*(e)-1| \leq \eps$ for all $v^* \in V$.
\end{enumerate}
\end{enumerate}

We also recall the notion of $\Gamma$-flat operators: a function $f$ from a topological space $\mathcal{T}$ to a metric space $M$ is called \emph{openly fragmented} if for every nonempty open subset $U \subseteq \mathcal{T}$ and any $\varepsilon>0$ there exists a nonempty open subset $V\subseteq U$ with $\diam (f(V)) <\varepsilon.$ Given a set $\Gamma\subseteq Y^*$ and $T \in \Lin(X,Y)$, we say that $T$ is \emph{$\Gamma$-flat} if $T^* \restricted_{\Gamma}\colon (\Gamma, w^*) \to X^*$ is openly fragmented. It is known that every Asplund operator $T\in\Lin(X,Y)$ is $\Gamma$-flat for every subset $\Gamma \subseteq B_{Y^*}$ (\cite[Example A]{CGKS2018}), and if $(\Gamma, w^*)$ is a discrete topological space, then every operator $T \in \Lin(X,Y)$ is $\Gamma_0$-flat for any $\Gamma_0\subseteq \Gamma$ (\cite[Example C]{CGKS2018}).

It was observed in \cite{CGKS2018} that $C(K)$ spaces, uniform algebras, and Banach spaces with property $\beta$ (introduced in \cite{Lin}) have ACK$_\rho$ structure. Besides, the property is stable by $\ell_\infty$-sums and by taking injective tensor products. Property quasi-$\beta$ was introduced in \cite{AAP2} as a generalization of property $\beta$ which still implies property B. As mentioned earlier, the notion of quasi-ACK structure is more general. Examples of Banach spaces having quasi-ACK structure include those with ACK$_\rho$ structure as well as those with property quasi-$\beta$ \cite[Remark 2.2]{CJ23}. Moreover, there is a Banach space that has quasi-ACK structure, while failing to have both ACK$_\rho$ structure and property quasi-$\beta$ \cite[Proposition 3.6]{CJ23}.

The next is the most general result that we are able to obtain. Even though the hypotheses seem to be very technical, we will be able to deduce many interesting consequences.

\begin{theorem}\label{theo:smooth_dense}
Let $X$ and $Y$ be Banach spaces such that $Y$ has quasi-$\ACK$ structure with a $1$-norming set $\Gamma \subseteq B_{Y^*}$ and $\rho\colon \Gamma \rightarrow [0,1)$. Let $T \in \Lin(X,Y)$ be a $\Gamma$-flat operator such that $T$ is $\Gamma_0$-flat for every $\Gamma_0 \subseteq \Gamma$. 
\begin{enumerate}
    \itemsep0.25em
    \item If $\SE(B_X)$ is dense in $X^*$, then given $\varepsilon>0$ there exists $R \in \ASE (X,Y)$ such that $\|R-T\| <\eps$; 
    \item If the set $\smo ({X^*}) \cap \NA(X)$ is dense in $X^*$ (in particular, if $X$ is strictly convex), then given $\eps >0$ there exists $R \in \Lin (X, Y)$ such that $\|R-T\| <\eps$, $\|R (x_0)\| = \|R\|$ for some $x_0 \in S_X$, and whenever a sequence $(x_n) \subseteq B_X$ satisfies that $\|Rx_n\| \to \|R\|$, there exists a subsequence $(x_{\sigma(n)}) \subseteq (x_n)$ such that $(x_{\sigma(n)})$ converges weakly to $\theta x_0$ for some $\theta \in \mathbb{T}$.
\end{enumerate} 
\end{theorem} 

We need some preliminary results. Recall the result given in \cite[Lemma 1.2]{JMR23} that if $T \in \Lin (X,Y)$ is such that $\|T\|= \|T^* y_0^*\|$ and $T^* y_0^* \in \SE (B_X)$ for some $y_0^* \in B_{Y^*}$, then $T$ attains its norm at a strongly exposed point; so $T \in \cl{\ASE(X,Y)}$. Let us present an analog version of this result considering more general elements in the set $\smo(X^*) \cap \NA(X)$ rather than those in $\SE(B_X)$. 

\begin{lem}\label{lem:smo+na}
Let $X$ and $Y$ be Banach spaces.
\begin{enumerate}
  \item If $T\in \ELL(X, Y)$ satisfies that If $\|T\| = \|T^* y_0^* \|$ with $y_0^*\in S_{Y^*}$ and $T^* y_0^* \in \smo({X^*})$ attains its norm at $x_0 \in B_X$, then, given $\varepsilon>0$, there is $S \in \Lin (X,Y)$ such that $\|S(x_0)\| = \|S\|$, $\|S-T\|<\varepsilon$, and whenever $(x_n) \subseteq B_X$ satisfies $\|Sx_n\| \to \|S\|$, there is a subsequence $(x_{\sigma(n)})$ of $(x_n)$ weakly convergent to $\theta x_0$ for some $\theta \in \mathbb{T}$.
  \item If $\mathcal{I}$ is an operator ideal contained in $\mathcal{DP}(X,Y)$ and $T\in \mathcal{I}(X, Y)$ satisfies $\|T\| = \|T^* y_0^* \|$ with $y_0^*\in S_{Y^*}$ and $T^* y_0^* \in \smo({X^*}) \cap \NA(X)$, then $T \in \overline{\mathcal{I}(X,Y) \cap \RSE (X,Y)}$.
\end{enumerate}
\end{lem}

\begin{proof}  Although the proof is similar to that of Lemma \ref{lem:strexp+density}, we present it here for the sake of completeness.
    
    (a). Let $x_0 \in B_X$ be such that $\|T^* y_0^*\| = (T^* y_0^*) (x_0)$. For simplicity, we may assume that $\|T\|=1$. Given $\varepsilon>0$, define $S \in \Lin (X,Y)$ by 
    \[
    S(x) = T(x) + \varepsilon \,(T^* y_0^*) (x) T x_0.
    \]
    Then $\|S\| \leq 1 + \varepsilon$, and 
    \[
    Sx_0 = (1+\varepsilon) Tx_0;
    \]
    hence $S$ attains its norm at $x_0$. Suppose that $\|Sx_n\| \to \|S\|$. By the definition of $S$, this implies that, passing to a subsequence, $(T^* y_0^*) (x_n) \to \theta \in \mathbb{T}$. Since $T^* y_0^*$ is a smooth point, we observe that $\bigl(\overline{\theta} x_n\bigr)$ converges to $x_0$ weakly and this completes the proof.

(b). Note from the proof of (a) that $S$ belongs to the same ideal $\mathcal{I}(X,Y)$ whenever $T$ belongs to $\mathcal{I}(X,Y)$. Moreover, if a sequence $(x_{\sigma(n)})$ converges weakly to $\theta x_0$ for some $\theta \in \mathbb{T}$, then the sequence $(Sx_{\sigma(n)} )$ converges to $\theta Sx_0$ in the norm topology; so $S \in \RSE(X,Y)$.
\end{proof}

\begin{proof}[Proof of Theorem~\ref{theo:smooth_dense}]
The proof is based on the arguments of the proof of \cite[Theorem 2.3]{CJ23}, now with the help of the above lemma. Let $\varepsilon>0$ and assume that $\|T\|=1$. By a result of Johannesen \cite[Theorem~5.8]{Lima1978}, we assume further that $T^* \in \Lin (Y^*, X^*)$ attains its norm at some $e^* \in \text{ext}(B_{Y^*})$. Let $\Gamma_{e^*} \subseteq \Gamma$ the corresponding set of quasi-ACK structure of $Y$ satisfying \ref{l1}--\ref{l3}. Find $y_0^* \in \Gamma_{e^*}$ and $x_0 \in S_X$ such that $|y_0^* (T(x_0))| > 1- \varepsilon$. Consider the set 
    \[
    U_0 := \{ y^* \in Y^*\colon |y^* ( T(x_0)) | > 1-\varepsilon\},
    \]
    which is weak$^*$-open. Since $T$ is $\Gamma_{e^*}$-flat from the hypothesis, there is weak$^*$-open subset $U_r \subseteq U_0$ with $U_r \cap \Gamma_{e^*} \neq \emptyset$ such that $\text{diam} (T^* (U_r \cap \Gamma_{e^*}))<\varepsilon$. Choose $y_1^* \in U_r \cap \Gamma_{e^*}$ and let $x_1^* = T^* (y_1^*)$. 

    Let $\Pi$ be either the set $\SE(B_X)$ or $\smo(X^*) \cap \NA(X)$. Suppose that the set $\Pi$ is dense in $X^*$. Then we can find $x_r^* \in \Pi$ with $\|x_r^*\|=1$ such that $\|x_r^* - x_1^*\| < \varepsilon$. Say $x_r^*$ attains its norm at $x_r \in S_X$. By quasi-ACK structure of $Y$, applying to $\varepsilon>0$ and $U:= U_r \cap \Gamma_{e^*}$, we can find $V \subseteq U$, $y_2^* \in V$, $e \in S_Y$ and $F \in \Lin (Y)$ satisfying \ref{l1'}--\ref{l6'}. Define $S \in \Lin (X,Y)$ by 
    \[
    S(x) := x_r^* (x) F(e) + ( 1 - \tilde{\varepsilon}) (\text{Id}_Y - F) T(x), \quad \text{ for every } x \in X,
    \]
    where $\tilde{\varepsilon}= \frac{5\varepsilon}{1-r_{e^*}+5\varepsilon}$ and $r_{e^*} = \sup_{y^* \in \Gamma_{e^*} } \rho (y^*) < 1$. We remark that $S^* (y_2^*) = x_r^*$. 

    Using conditions \ref{l1'}--\ref{l6'}, one can verify that $\|S-T\| \leq 5 \varepsilon + 2 \tilde{\varepsilon}$ and 
    \[
    \|S\| = \|Sx_r\| = |(S^* y_2^*)(x_r)| = |x_r^* ( x_r) | = 1
    \]
    (for a detailed proof, we refer to the proof of \cite[Theorem 2.3]{CJ23}).

    If the set $\Pi$ were $\SE(B_X)$, then the point $x_r$ should be a strongly exposed point. This shows that there exists $R \in \ASE(X,Y)$ so that $\|R-S\| < \varepsilon$ \cite[Proposition 3.14]{cgmr2020}. This proves (a). Next, suppose that $\Pi = \smo(X^*) \cap \NA(X)$. In this case, $S^* y_2^*  \in \smo (X^*)$ attains its norm at $x_r$ with $\|S \| = \|S^* y_2^*\|$; hence Lemma \ref{lem:smo+na}.(a) yields the result (b). 
\end{proof} 

\begin{remark} Let $X$ be a Banach space. 
\begin{enumerate}
    \itemsep0.25em 
    \item By definition, $\SE(B_X) \subseteq \smo(X^*)\cap\NA(X)$. As expected, this inclusion may be strict (even for reflexive Banach spaces). For instance, if we denote $(\ell_2, \| \cdot\|_W)$ the renorming of $\ell_2$ due to Smith \cite{Smith}, then it is known that the norm of $(\ell_2, \| \cdot\|_W^*)$ is G\^{a}teaux differentiable but not Fr\'echet differentiable (see the paragraph after \cite[Theorem 5.6.12]{megginson}). As an example of non-reflexive Banach spaces, consider $X = c$, that is, the space of all convergent sequences endowed with the supremum norm. Then $\SE(B_X) = \emptyset$ while the set $\smo(X^*) \cap \NA(X)$ is dense in $X^* (=\ell_1)$. 
    Finally, note that if $X=c_0$, then the set $\smo(X^*)\cap\NA(X)$ turns to be the empty set. So, the results with the assumption on the denseness of $\smo(X^*)\cap\NA(X)$ cannot be applied to $X=c_0$. 
    However, the denseness of certain $\RSE$ operators on $c_0$ can be derived by Corollary \ref{corollary--bigoneforfiniterank-compact-fordomainspaces}.
    \item We have the following.
\[
\begin{tikzcd}
X \text{ is strictly convex} \arrow[d, "(b)", Leftrightarrow] & X^* \text{ is smooth} \arrow[l, "(a)"', Rightarrow]     & \\ 
 \NA(X)\subseteq\smo(X^*) \arrow[r, "(c)", Rightarrow] & \NA(X)\cap\smo(X^*) \text{ is dense}
\end{tikzcd}
\]
The implication (a) is quite well known (and easy to prove), and the converse of (a) is not true. The equivalence (b) is a consequence of the Hahn-Banach theorem and \v{S}mulyan's lemma. The implication (c) is obvious. 
\end{enumerate}    
\end{remark}

As mentioned at the beginning of this section, Banach spaces with property quasi-$\beta$ have quasi-ACK structure. More precisely, if $Y$ has property quasi-$\beta$ witnessed by the set $A=\{x_\alpha^*\colon \alpha \in \Lambda\} \subseteq S_{X^*}$, $\{x_\alpha\colon \alpha \in \Lambda\}\subseteq S_X$, and a real-valued function $\rho$ on $A$, then $Y$ has quasi-ACK structure with respect to the $1$-norming set $\Gamma:=A$ and the above function $\rho$ (see \cite[Remark 2.2]{CJ23}). Besides, still in the case when $Y$ has property quasi-$\beta$, every operator $T \in \Lin (X,Y)$ is $\Gamma_0$-flat for any $\Gamma_0 \subseteq \Gamma$ since $(\Gamma,w^*)$ is a discrete topological space. This discussion allows us to give the following consequences of Theorem~\ref{theo:smooth_dense}.

\begin{cor}\label{cor:smooth_dense}
Let $X$ and $Y$ be Banach spaces, and let an operator ideal $\mathcal{I}$ be given. If $\SE(B_X)$ is dense in $X^*$, $Y$ has quasi-$\ACK$ structure, and $\mathcal{I}$ is contained in the operator ideal $\mathcal{A}$ of Asplund operators, then 
     \[
    \mathcal{I}(X,Y) \subseteq \overline{ \mathcal{I}(X,Y) \cap \ASE(X,Y) }
     \]
\end{cor} 

\begin{proof}
Notice from the proof of Theorem \ref{theo:smooth_dense} that when $T$ belongs to $\mathcal{I}(X,Y)$, so does $R$. Moreover, as mentioned earlier, $T$ is $\Gamma$-flat for every subset $\Gamma \subseteq B_{Y^*}$, so the denseness follows from (a) of Theorem \ref{theo:smooth_dense}. 
\end{proof}

\begin{cor}\label{cor:smooth_dense2}
Let $X$ and $Y$ be Banach spaces, and an operator ideal $\mathcal{I}$ be given. If the set $\smo (X^*) \cap \NA(X)$ is dense in $X^*$, $Y$ has property quasi-$\beta$, and $\mathcal{I}$ is contained in $\DP(X,Y)$, then 
\[
\mathcal{I}(X,Y) \subseteq \cl{ \mathcal{I} (X,Y) \cap \RSE (X,Y) }.
\]
\end{cor} 

\begin{proof}
As $Y$ is assumed to have property quasi-$\beta$, we have that $Y$ has quasi-ACK structure with a certain $1$-norming set $\Gamma \subseteq B_{Y^*}$ and a real-valued function on $\Gamma$. Again, as mentioned above, every member in $\Lin (X,Y)$ is $\Gamma_0$-flat for any subset $\Gamma_0 \subseteq \Gamma$. Using the definition of Dunford-Pettis property as in Lemma \ref{lem:smo+na}.(b), we conclude the result. 
\end{proof}

An interesting particular case is given when $Y$ is a closed subspace of $c_0(\Gamma)$, as then $Y$ has property quasi-$\beta$ \cite[Example 3.2]{JMR23}.

\begin{cor}
Let $X$ be a Banach space such that $\smo (X^*) \cap \NA(X)$ is dense in $X^*$ and let $Y$ be a closed subspace of $c_0(\Gamma)$ for some set $\Gamma$. If $\mathcal{I}$ is an operator ideal  contained in $\DP$, then 
\[
\mathcal{I}(X,Y) \subseteq \cl{ \mathcal{I} (X,Y) \cap \RSE (X,Y) }.
\]
\end{cor} 

This corollary can be applied of course to compact operators.

\begin{cor}\label{cor:smo_dense_c0gamma}
Let $X$ be a Banach space such that $\smo (X^*) \cap \NA(X)$ is dense in $X^*$ and let $Y$ be a closed subspace of $c_0(\Gamma)$ for some set $\Gamma$. Then 
\[
\comp(X,Y) \subseteq \cl{ \comp (X,Y) \cap \RSE (X,Y) }.
\]    
\end{cor}

Let us remark that the condition on the domain space in Corollary~\ref{cor:smo_dense_c0gamma} is essential, as shown by Proposition~\ref{prop:no-universal-RSE-B-compact} when applied to subspaces of $c_0(\Gamma)$ that fail the approximation property.

The next consequence follows as a combination of Corollary \ref{cor:smooth_dense2} and the arguments of Johnson--Wolfe in \cite{JW79}. 

\begin{cor}\label{cor:L1_predual}
Let $X$ be a Banach space such that $\smo (X^*) \cap \NA(X)$ is dense in $X^*$. Then, for any $L_1$-predual space $Y$, we have 
$$
\mathcal{K}(X,Y) = \overline{ \mathcal{K}(X,Y) \cap \RSE(X,Y)}.
$$ 
\end{cor} 

\begin{proof}
Let $T \in \mathcal{K}(X,Y)$ and $\varepsilon>0$ be given. Arguing like in \cite[Lemma~3.4]{JW79} we can find a subspace $E$ of $Y$ which is isometric to $\ell_{\infty}^m$ for some $m\in \N$ and a projection $P\colon Y\to E \subseteq Y$ with $\|P\|=1$ such that $\|T-PT\| < \varepsilon$. 
As $\ell_\infty^m$ has property $\beta$, there is $S \in \RSE(X,Y)$ such that $\|PT-S\|<\eps$ by Corollary \ref{cor:smooth_dense2}, hence $\|T-S\| < 2\varepsilon$.
\end{proof} 

Finally, let us remark that, even though it is straightforward to show that $\NA(X,\ell_\infty^m)$ is dense in $\ELL(X,\ell_\infty^m)$ for every $m\in \N$ and every Banach space $X$, we do not know if the set $\RSE(X,\ell_\infty^m)$ is also dense without the assumption that $\smo(X^*)\cap \NA(X)$ is dense in $X^*$.

\section{RSE Analogues of Zizler's and Lindenstrauss's Results on Adjoint Operators}\label{sec:Zizler}
We start by showing that there is no analogous results for $\RSE$ operators to Zizler's \cite{Zizler} and Lindenstrauss' \cite{Lin} on the denseness on the set of operators whose adjoint or biadjoint attain their norm.

\begin{example}
The sets 
\begin{enumerate}
\itemsep0.25em
  \item $\{T\in \ELL(\ell_1,\ell_1)\colon T^*\in \RSE(\ell_\infty,\ell_\infty)\}$, 
  \item $\{T\in \ELL(c_0,c_0)\colon T^{**}\in \RSE(\ell_\infty,\ell_\infty)\}$,
\end{enumerate}
are not dense in $\ELL(\ell_1,\ell_1)$ and $\ELL(c_0,c_0)$, respectively.
\end{example}

\begin{proof}
For (a), observe that, since the set of isomorphisms is an open set, if the set in (a) were dense, it would exists an isomorphism $S\in \ELL(\ell_1,\ell_1)$ such that $S^*\in \RSE(\ell_\infty,\ell_\infty)$; being $S^*$ also an isomorphism, Remark~\ref{rem:monomorf} would give that $S^*\in \ASE(\ell_\infty,\ell_\infty)$, but this is impossible since the unit ball of $\ell_\infty$ fails to contain strongly exposed points. The argument for (b) is completely analogous.
\end{proof}

It is then clear that we need conditions on $X$, $Y$, or on the involved operator to approximate it by operators whose adjoint are $\RSE$. The most general result we are able to present in this line deals with operators with closed range and such that their image of the unit ball have the Asplund property. Let us first introduce the needed definitions.

Let $D \subseteq X$ be a bounded subset and $F\colon X \rightarrow \mathbb{R}$ any function. Then $F$ is \textit{$D$-differentiable at} $x \in X$ if there exists $f \in X^*$ such that 
\[
\lim_{t \rightarrow 0+}  \sup_{d \in D} \left| \frac{ F(x+td)- F(x)}{t} - f (d) \right| = 0.
\]
In this case, $f$ is called the \textit{$D$-gradient} of $F$ at $x$. A bounded subset $D$ of $X$ is said to have the \textit{Asplund property} if each convex continuous function $F \colon X\rightarrow \mathbb{R}$ is $D$-differentiable on a residual subset of $X$. 
{It is known \cite[Theorem 5.3.5]{bourgin} that $T \in \Lin (X,Y)$ is an Asplund operator (i.e., $T$ factors through an Asplund space) if and only if the set $T(B_X)$ has the Asplund property}. Moreover, it is equivalent to that $T^* (B_{Y^*})$ is an RNP set \cite[Theorem 5.2.3]{bourgin}. It is also known that if $T$ is an isomorphism, then $D \subseteq X$ has the Asplund property if and only if $T(D) \subseteq Y$ has the Asplund property \cite[Lemma 5.3.1]{bourgin}. In particular, if $D \subseteq X$ has the Asplund property and $Z$ is a closed subspace containing the set $D$, then $D \subseteq Z$ has the Asplund property. 

\begin{prop}\label{prop:closed_range}
Let $X$ and $Y$ be Banach spaces, and $T \in \Lin (X,Y)$ be an operator with closed range. Suppose that {$T(B_X)$} has the Asplund property. Let $\Phi\colon T^* (B_{Y^*}) \rightarrow \mathbb{R}$ be a bounded above and upper semicontinuous function. Then, given $\eps >0$ and $x_1 \in X$, there exists $x_0 \in X$ such that $\Phi +   x_0$ strongly exposes $T^* (B_{Y^*})$ and $\|Tx_0 - Tx_1 \| <\eps$. 
\end{prop}

The proof of this result is based on \cite[Theorem 2.6]{AAGM2010-variational}, with appropariate modifications to fit our setting. To make the proof self-contained, we recall the following lemma from that paper.

\begin{lemma}[\mbox{\cite[Lemma 2.4]{AAGM2010-variational}}]\label{lem:AAGC_lemma24}
    Let $X$ be a Banach space, $W$ a dense subset of $X \times \mathbb{K}$ such that $\mathbb{R}^+ W \subseteq W$ and $\phi\colon X \rightarrow \mathbb{R}$ is a convex continuous function satisfying $\phi(0)<0$. Then, the subset $W \cap \operatorname{Graph}(\phi)$ is dense in $\operatorname{Graph}(\phi)$.
\end{lemma}

Let us point out that \cite[Lemma 2.4]{AAGM2010-variational} was origianlly stated for the case $\mathbb{K}=\mathbb{R}$, however, the (almost) same proof applies to Lemma \ref{lem:AAGC_lemma24} as well.  

\begin{proof}[Proof of Proposition~\ref{prop:closed_range}]
For simplicity, let $Z := T(X)$ and $E:= T^* (B_{Y^*}) \subseteq X^*$. Let $\Phi\colon E \rightarrow \mathbb{R}$ be a bounded above and upper semicontinuous function, and consider $\sigma \colon Z \times \mathbb{K} \rightarrow \mathbb{R}$ by 
\[
\sigma (y,s) := \sup \left \{ \re \left\langle \frac{t}{1+M-\Phi(T^* y^*)} (y^*, -1), \, (y,s) \right\rangle \colon t \in [0,1], \,y^* \in B_{Y^*} \right\},
\]
where $M := \sup \Phi (E)$.

It is not difficult to check that $T(B_X) \times B_{\mathbb{K}} \subseteq Y \times \mathbb{K}$ has the Asplund property; therefore $T(B_X) \times B_{\mathbb{K}} \subseteq Z \times \mathbb{K}$ has the Asplund property. Since $\sigma$ is continuous and convex, we observe that $\sigma$ is $(T(B_X)\times B_{\mathbb{K}})$-differentiable on a $G_\delta$-dense subset $W$ of $Z \times \mathbb{K}$. Note that $\sigma (r (y,s)) = r \sigma (y,s)$ for every $(y,s)\in Z\times \mathbb{K}$ and $r\in\mathbb{R}^+$; so $\mathbb{R}^+ W \subseteq W$.

Put 
\[
H:= \left \{ \frac{t}{1+M-\Phi(T^* y^*)} (y^*, -1)\colon t \in [0,1], \, y^* \in B_{Y^*} \right\}.
\]
Observe that $H$ is closed, and $\sigma (y,s) = \sup \{ \re \langle (u^*, s^*), (y,s) \rangle \colon (u^*, s^*) \in H\}$.
\vspace{0.25em}

\textit{Step 1}: If $\sigma$ is $(T(B_X)\times B_{\mathbb{K}})$-differentiable at $(z_0, r_0) \in W$ with $(T(B_X) \times B_{\mathbb{K}})$-gradient $(z_0^*, r_0^*)$, then the element $(z_0, r_0)$ strongly exposes $H$ at $(z_0^*, r_0^*)$ in the following sense: 
\begin{align}
    \text{if $(u_n^*,s_n^* ) \in H$ satisfies that } \re \langle (u_n^*, s_n^*) ,  &(z_0, r_0) \rangle \to \sigma (z_0,r_0), \nonumber \\
    &\text{ then $\|(u_n^*, s_n^*)-(z_0^*, r_0^*) \|_{B_Z\times B_\mathbb{R}} \to 0$. } \tag{$*$} \label{eq:strong_exposure}
\end{align}
To prove the claim, given $\eps >0$, find $\delta >0$ such that 
\[
0<t \leq \delta \implies \sup_{ u \in T(B_X) \times B_\mathbb{K} } \left| \frac{ \sigma( (z_0,r_0) + t u )-\sigma ( z_0, r_0) }{t} - (z_0^*, r_0^*)(u) \right| < \eps. 
\]
Fix $\alpha = \eps \delta$
and suppose that $ (u^*, s^*) \in H$ satisfies that 
$$ \re \langle (u^*, s^*) ,  (z_0, r_0) \rangle > \sigma (z_0,r_0) - \alpha.$$
Then, for $(Tx, s)  \in T(B_X) \times B_{\mathbb{K}}$, we have 
\begin{align}
    \re \langle &(u^*, s^*) - (z_0^*, r_0^*), (Tx, s) \rangle \nonumber \\
    &= \frac{1}{\delta} \re \left[ \langle (u^*, s^*) ,  (z_0, r_0) + \delta (Tx, s) \rangle - \langle (u^*, s^*) ,  (z_0, r_0) \rangle - \langle (z_0^*, r_0^*), \delta (Tx, s) \rangle \right] \nonumber \\ 
    &\leq \frac{1}{\delta} \left[ \sigma ( (z_0, r_0) + \delta (Tx, s) )   - \re \langle (u^*, s^*) ,  (z_0, r_0) \rangle - \re  \langle (z_0^*, r_0^*), \delta (Tx, s) \rangle \right] \nonumber \\
    &\leq \frac{1}{\delta} \left[ \varepsilon \delta  + \sigma (z_0,r_0) - \re \langle (u^*, s^*) ,  (z_0, r_0) \rangle \right] < \varepsilon + \frac{\alpha}{\delta} = 2 \varepsilon. \nonumber 
\end{align}
Since $\theta ( T(B_X) \times B_{\mathbb{K}} ) = T(B_X)\times B_{\mathbb{K}}$ for all $\theta \in \mathbb{K}$ with $|\theta|=1$, we conclude that $\| (u^*, s^*) - (z_0^*, r_0^*)\|_{T(B_X) \times B_{\mathbb{K}}} < \eps$. As $T$ has a closed range, by open mapping theorem, we have that $\| (u^*, s^*) - (z_0^*, r_0^*) \|_{B_Z \times B_{\mathbb{R}} } < C \eps$ for some constant $C>0$. 
As a consequence, we have that every element $(z_0, r_0)$ in $W$ strongly exposes $H$ in the sense of \eqref{eq:strong_exposure} (at some $(z_0^*, r_0^*)$ where $(z_0^*, r_0^*)$ is the corresponding $(T(B_X) \times B_{\mathbb{K}})$-gradient). 
 
Next, define $\phi\colon Z = T(X) \rightarrow \mathbb{R}$ by 
\[
\phi (y):= \sup \{ \re \{  \Phi (T^*y^*) + y^* (y) - M - 1 \} \colon y^* \in B_{Y^*} \}, \quad \forall y \in Z. 
\]
Note that $\phi$ is continuous, convex, and $\phi(0)=-1<0$. 
By applying Lemma \ref{lem:AAGC_lemma24} to $W \subseteq Z \times \mathbb{K}$ and $\phi:Z\to\mathbb{R}$, we observe that 
the elements $(y, \phi (y)) \in Z \times \mathbb{R}$ strongly exposing $H$ in the sense of \eqref{eq:strong_exposure} form a dense subset in the graph of $\phi$.
\vspace{0.25em}

\textit{Step 2}: Given $x_0 \in X$, if $(Tx_0, \phi(Tx_0)) \in Z \times \mathbb{R}$ strongly exposes $H$ in the sense of \eqref{eq:strong_exposure} at 
\[
 \frac{t_0}{1+M-\Phi(T^* y_0^*)} (y_0^*, -1)
\]
for some $t_0 \in [0,1]$ and $y_0^* \in B_{Y^*}$, then $\Phi + \re x_0$ strongly exposes the set $E=T^* (B_{Y^*})$ at $T^* y_0^*$.
\vspace{0.25em}
 
To this end, note first that $t_0$ can be assumed to be $1$ and that 
\[
 1 \geq \frac{\re y_0^* (Tx_0) - \phi (Tx_0) } {1+M-\Phi(T^* y_0^*)} \geq \sup_{y^* \in B_{Y^*} } \frac{\re y^* (Tx_0) - \phi (Tx_0) } {1+M-\Phi(T^* y^*)} = 1. 
\]
This shows that $\Phi + \re x_0$ attains its maximum at $T^* y_0^* \in E$ with 
\[
(\Phi + \re x_0)(T^* y_0^*) = 1+M+\phi(Tx_0). 
\]
If $(y_n^*) \subseteq B_{Y^*}$ satisfies $(\Phi + \re x_0) (T^* y_n^*) \rightarrow 1+M+\phi(Tx_0)$, then this would imply that 
\begin{align*}
0 &\leq 1 -  \re \left \langle (Tx_0, \phi (Tx_0) ), \frac{(y_n^*, -1)}{1+M-\Phi(T^*y_n^*)} \right \rangle \\
&= \frac{ 1 + M - \Phi (T^*y_n^*) -\re y_n^* (Tx_0)+\phi(Tx_0) }{1+M-\Phi(T^*y_n^*)}\\
&\leq   1 + M - \Phi (T^* y_n^*)-\re y_n^* (Tx_0)+\phi(Tx_0)  \to 0.
\end{align*}
Using that $(Tx_0, \phi(Tx_0))$ strongly exposes $H$ in the sense of \eqref{eq:strong_exposure},  
\[
\left\| \left(  \frac{1}{1+M-\Phi(T^* y_n^*)} (y_n^*, -1) \right) - \frac{1}{1+M-\Phi(T^* y_0^*)} (y_0^*, -1)  \right\|_{B_Z \times B_\mathbb{K}} \to 0.
\]
This, in particular, shows that $\|y_n^* - y_0^*\|_{B_Z} \to 0$; hence $\|T^* y_n^* - T^* y_0^*\| \to 0$. 

\textit{Step 3}: 
Recall from the end of Step 1 that the set 
\[
\{ ( Tx, \phi (Tx) ) \text{ strongly exposing $H$ in the sense of \eqref{eq:strong_exposure}}  \colon x \in X\}
\]
is dense in $\{ (Tx, \phi(Tx ) ) \colon x \in X \}$. 
Therefore, given $x_1 \in X$ and $\eps >0$, we can find $x_0 \in X$ such that $\|Tx_1 - Tx_0\| < \varepsilon$ and $(Tx_0,\phi(Tx_0))$ strongly exposes $H$ in the sense of \eqref{eq:strong_exposure}. By Step 2, we then obtain that $\Phi + \re x_0$ strongly exposes the set $E$ at some $T^* y_0^*$ with $y_0^* \in B_{Y^*}$. 
\end{proof}

The result can be directly applied to Asplund operators to get the following nice consequence. 

\begin{cor}\label{cor:Asplund_closed_range}
Let $X$ and $Y$ be Banach spaces, and let $\mathcal{I}$ be an operator ideal contained in $\mathcal{A}$. If $T \in \mathcal{I}(X,Y)$ has closed range, then given $\eps >0$ there exists $S \in \mathcal{I} (X,Y)$ such that $\|S-T\| < \eps$ and $S^* \in \RSE(Y^*, X^*)$.
\end{cor}

\begin{proof}
    Let $\varepsilon>0$ be given and consider $\Phi := \| \cdot \|_{X^*} \colon T^* (B_{Y^*}) \rightarrow \mathbb{R}$. By Proposition \ref{prop:closed_range}, we can find $x_0 \in S$ such that $\Phi + \text{Re}\, x_0$ strongly exposes $T^* (B_{Y^*})$ at some $T^* y_0^*$ and $\|Tx_0 \| < \varepsilon$. Define $R \in \Lin (Y^*, X^*)$ by 
    \[
    Ry^* = T^* y^* + \frac{T^*y_0^*}{\|T^* y_0^*\|} y^* (Tx_0), \quad y^* \in Z^*.
    \]
    Notice that $R$ is weak-star to weak-star continuous, so $R=S^*$ for some $S \in \Lin (X,Y)$. Since $S-T$ is a rank-one operator, we have $S \in \mathcal{I}(X,Y)$. Moreover, $\|R\| = (\Phi+\re x_0) (T^* y_0^*)$. Now, 
    if $(y_n^*) \subseteq B_{Y^*}$ is such that $\|Ry_n^*\| \rightarrow \|R\|$, then 
    $$
       \|R\| = [\Phi + \re x_0] (T^*y_0^*) \geq [\Phi + \text{Re}\, x_0] (T^* y_n^*)  \geq \|Ry_n^*\|  \to \|R\|; 
     $$ 
     so $(T^* y_n^*)$ converges to $T^* y_0^*$. It follows that 
    \[
    R y_n^* \rightarrow T^* y_0^* + \frac{T^*y_0^*}{\|T^* y_0^*\|} y_0^* (Tx_0) = R y_0^* \text{ and } \|R y_0^*\| = \|R\|;
    \]
    so $R \in \RSE(Y^*, X^*)$. 
\end{proof}

Observe that $T \in \mathcal{A}(X,Y)$ if and only if $T^* \in \mathcal{SRN}(Y^*,X^*)$ \cite[Theorem 5.2.11]{bourgin}. In this regard, Corollary \ref{cor:Asplund_closed_range} can be seen as a dual version of the denseness result of $\mathcal{SRN} \cap \UQNA$ in $\mathcal{SRN}$ proved in \cite[Theorem 3.1]{CCJM}.

If we deal with finite-rank operators, then the Asplundness condition and the closed rank condition automatically hold, so we get the following immediate consequence of Corollary  \ref{cor:Asplund_closed_range}. 

\begin{corollary}
Let $X$ and $Y$ be Banach spaces. Then, 
$$
\mathcal{FR}(X,Y) \subseteq \overline{ \{ R \in \mathcal{FR} (X,Y) \colon R^* \in \RSE (Y^*, X^*) \}}.
$$
If, moreover, $X^*$ or $Y$ has the approximation property, then
$$
\comp(X,Y) \subseteq \overline{ \{ R \in \mathcal{FR} (X,Y) \colon R^* \in \RSE (Y^*, X^*) \} }.
$$
\end{corollary}

Let us observe that the adjoint of a compact operator is always norm-attaining, and that compact operators can be always approximated by $\UQNA$ operators \cite[Corollary 3.11]{CCJM}, but it is not clear whether when one starts with the adjoint of a compact operator, the $\UQNA$ approximating operators can also be taken as adjoint operators, which is exactly what the previous corollary ensures. 

Our last result in this section deals with the relation of the $\UQNA$ and $\RSE$ operators and Fr\'{e}chet differentiability points, similar to the known relation in the case of $\ASE$ operators. Recall that $T$ is a point of Fr\'{e}chet differentiability in $\Lin (X,Y)$ if and only if $T \in \ASE(X,Y)$ at $x_0 \in S_X$ and $Tx_0$ is a point of Fr\'{e}chet differentiability in $Y$ (see \cite[Proposition 1.6]{JMR23} or \cite[Theorem 3.1]{H1975}). Moreover, $T$ is a Fr\'{e}chet differentiability point if and only if $T^*$ is a Fr\'{e}chet differentiability point. We can also relate the notions of $\UQNA$ and $\RSE$ to Fr\'{e}chet differentiability in a similar way as follows. 
 
\begin{prop}\label{lem:dual}
Let $X$ and $Y$ be Banach spaces and $T \in \Lin (X,Y)$. Then  
\begin{enumerate}
\itemsep0.25em
\item $T \in \UQNA (X,Y)$ towards $\|T\| y_0$ for some Fr\'{e}chet differentiability point $y_0 \in S_Y$ if and only if $T^* \in \ASE (Y^*, X^*)$ at $y_0^* \in S_{Y^*}$ with the following property
\[
\text{if } (z_n) \subseteq B_X \text{ satisfies } \re (T^*y_0^*)(z_n) \rightarrow \|T\|, \text{ then } Tz_n \rightarrow \|T\| y_0. 
\]
\item $T \in \RSE (X,Y)$ towards $Tx_0$ for some Fr\'{e}chet differentiability point $Tx_0 \in \|T\| S_Y$ with $x_0 \in S_X$ if and only if $T^* \in \ASE (Y^*, X^*)$ at $y_0^* \in S_{Y^*}$ with the following property: there exists $x_0 \in S_X$ so that 
\[
\text{if } (z_n) \subseteq B_X \text{ satisfies } \re (T^*y_0^*)(z_n) \rightarrow \|T\|, \text{ then } Tz_n \rightarrow Tx_0. 
\]
\end{enumerate} 
\end{prop} 

\begin{proof}
We only prove (b), being (a) completely analogous. $(\Rightarrow)$: Suppose $\|T^*y_n^* \| \rightarrow \|T^*\|$. Find $x_n \in B_X$ so that $y_n^* (Tx_n) \rightarrow \|T\|$. Then $Tx_n \rightarrow \theta Tx_0$ passing to a subsequence for some $\theta \in \mathbb{T}$. Thus, $\theta  y_n^* (Tx_0) \rightarrow \|T\|$; This implies that $(\theta y_n^*)$ converges to $y_0^*$ where $y_0^* \in S_Y^*$ satisfies that $y_0^* (Tx_0) = \|T\|$. Thus, $T^* \in \ASE(Y^*, X^*)$ and $T^* y_0^*$ satisfy the desired property. 

$(\Leftarrow)$: If $\|Tx_n\| \rightarrow \|T\|$, then we can find $y_n^* \in S_{Y^*}$ so that $(T^* y_n^*)(x_n) \rightarrow \|T\|$; hence $y_n^* \rightarrow \theta y_0^*$, passing to a subsequence, for some $\theta \in \mathbb{T}$. This shows that $ (T^*y_0^*)(\theta x_n) \rightarrow \|T\|$; hence $Tx_n \rightarrow \theta Tx_0$. So, $T \in \RSE (X,Y)$. Finally, let $z_n^* \in B_Y^*$ so that $z_n^* (Tx_0) \rightarrow \|T\| =y_0^* (Tx_0)$. This would imply that $(z_n^*)$ converges to $y_0^*$; so $Tx_0$ is a Fr\'{e}chet differentiability point.
\end{proof} 

\section*{Acknowledgments}

We are grateful to Rafael Pay\'{a}, Jos\'{e} Rodr\'{\i}guez, and Abraham Rueda Zoca for fruitful discussions on this topic.

M.\ Jung was supported by June E Huh Center for Mathematical Challenges (HP086601)
at Korea Institute for Advanced Study. H.\ del R\'io, A.\ Fovelle, and M.\ Mart\'in were partially supported by MICIU/AEI/10.13039/501100011033 and ERDF/EU through the grant PID2021-122126NB-C31, and by ``Maria de Maeztu'' Excellence Unit IMAG, funded by MICIU/AEI/10.13039/501100011033 with reference CEX2020-001105-M.

\bibliographystyle{alpha}

\end{document}